\documentclass[12pt,oneside]{amsart}
\usepackage[top=1in, left=1in, right=1in, bottom=1in]{geometry}
\allowdisplaybreaks[1]
\usepackage{latexsym}
\usepackage{amsmath}
\usepackage{amssymb}
\usepackage{amscd}
\usepackage{array}
\usepackage{hhline}
\usepackage{color}
\usepackage{hyperref}
\hypersetup{
    colorlinks,
    citecolor=black,
    filecolor=black,
    linkcolor=black,
    urlcolor=black
}
\usepackage{bbold}
\usepackage[all]{xy}
\usepackage{graphicx}
\newtheorem{definition}{Definition}[section]
\newtheorem{theorem}[definition]{Theorem}
\newtheorem{proposition}[definition]{Proposition}
\newtheorem{remark}[definition]{Remark}
\newtheorem{lemma}[definition]{Lemma}
\newtheorem{corollary}[definition]{Corollary}

\newtheorem{example}[definition]{Example}

\def\C{{\mathbb{C}}}

\def\N{{\mathbb{N}}}
\def\Z{{\mathbb{Z}}}

\def\deg{{\mathrm{deg}}}
\def\dim{{\mathrm{dim}}}

\def\rank{{\mathrm{rank\;}}}
\def\End{{\mathrm{End}}}

\def\Hom{{\mathrm{Hom}}}
\def\HOM{{\mathrm{HOM}}}

\def\EXT{{\mathrm{EXT}}}

\def\ui{{\text{$\underline{i}$}}}

\def\oE{{\widehat{E}}}

\def\ol{\overline{\lambda}}

\def\l{{\lambda}}

\def\idm{\text{\bfseries 1}}

\def\u{{\mathcal{U}}}
\def\e{{\mathcal{E}}}
\def\id{{\mathrm{Id}}}

\def\HMF{{\mathrm{HMF}}}

\def\ostimes{{\,\otimes\hspace{-0.7em}\raisebox{-0.5ex}{${}_{{}_{S}}$}\,}}

\newcommand{\kaku}[1]{\txt{\begin{picture}(0,0)(0,0)\special{pn 8}\special{pa 100 100}\special{pa -100 100}\special{pa -100 -100}\special{pa 100 -100}\special{pa 100 100}\special{fp}\put(0,0){\makebox(0,0){\emph{$#1$}}}\end{picture}}}
\newcommand{\skaku}[1]{\txt{\begin{picture}(0,0)(0,0)\special{pn 8}\special{pa 80 80}\special{pa -80 80}\special{pa -80 -80}\special{pa 80 -80}\special{pa 80 80}\special{fp}\put(0,0){\makebox(0,0){\emph{$#1$}}}\end{picture}}}
\newcommand{\wkaku}[1]{\txt{\begin{picture}(0,0)(0,0)\special{pn 8}\special{pa 130 100}\special{pa -130 100}\special{pa -130 -100}\special{pa 130 -100}\special{pa 130 100}\special{fp}\put(0,0){\makebox(0,0){\emph{$#1$}}}\end{picture}}}
\newcommand{\wwkaku}[1]{\txt{\begin{picture}(0,0)(0,0)\special{pn 8}\special{pa 145 100}\special{pa -145 100}\special{pa -145 -100}\special{pa 145 -100}\special{pa 145 100}\special{fp}\put(0,0){\makebox(0,0){\emph{$#1$}}}\end{picture}}}

\newcommand{\btime}[1]{\boxtimes\hspace{-1em}\raisebox{-0.8ex}{${}_{#1}$}\,\,}

\newcommand{\U}[1]{\dot{\mathbf U}_q(\mathfrak{sl}_{#1})}

\newcommand{\Uv}[1]{\dot{\mathbf U}_q(\mathfrak{sl}_{#1})}

\title{$\mathfrak{sl}_N$-Web categories}
\author{Marco Mackaay and Yasuyoshi Yonezawa}
\thanks{Both authors were supported by the FCT - Fundac\~{a}o para a 
Ci\^{e}ncia e a Tecnologia, through project number PTDC/MAT/101503/2008, 
New Geometry and Topology.}
\date{\today}
\begin{document}

\begin{abstract}
In this paper we show how the colored Khovanov-Rozansky 
$\mathfrak{sl}_N$-matrix factorizations, due to Wu~\cite{wu} and 
Y.Y~\cite{yo1,yo2}, can be used 
to categorify the quantum skew Howe duality 
defined by Cautis, Kamnitzer and Morrison in~\cite{ckm}. 
In particular, we define web categories and $2$-representations of 
Khovanov and Lauda's categorical quantum $\mathfrak{sl}_m$ 
on them. We also show that this implies that each such web category is 
equivalent to the category of finite-dimensional graded projective modules over 
a certain level-$N$ cyclotomic KLR-algebra. 
\end{abstract}

\maketitle

\paragraph*{Acknowledgements}
M.~M thanks Bruce Fontaine, Joel Kamnitzer, Mikhail Khovanov and Ben Webster 
for helpful exchanges of emails and discussions. 
Y.Y thanks Yoshiyuki Kimura, Naoya Enomoto and Catharina Stroppel for many helpful discussions.
We thank Daniel Tubbenhauer for comments on an earlier version of this paper. 
\tableofcontents
%
%
%
%
\section{Introduction}\label{intro}

Recently Cautis, Kamnitzer and Morrison~\cite{ckm} found a complete set of 
relations on colored $\mathfrak{sl}_N$-webs. We recall that these webs 
represent intertwiners between tensor products of fundamental 
${\mathbf U}_q(\mathfrak{sl}_N)$-representations. The main ingredient 
in~\cite{ckm} was a diagrammatic version of quantum skew Howe duality, 
which shows 
that ${\mathbf U}_q(\mathfrak{sl}_m)$ acts on 
$\mathfrak{sl}_N$-web spaces, where $m$ and $N$ are distinct non-negative 
integers. 

In this paper we show how the colored $\mathfrak{sl}_N$-matrix factorizations 
can be used to categorify Cautis, Kamnitzer and Morrison's results. These 
matrix factorizations are due to Wu~\cite{wu} and Y.Y.~\cite{yo1,yo2} and 
generalize Khovanov and Rozansky's~\cite{kr} matrix factorizations 
in their groundbreaking work on $\mathfrak{sl}_N$-link homologies.  
\vskip0.5cm 
To be a bit more precise, let $N\geq 2$ and $m,d\geq 0$ be arbitrary 
integers. We first define a $2$-functor $\Gamma_{m,d,N}$ from 
Khovanov and Lauda's categorified quantum $\mathfrak{sl}_m$, 
defined in~\cite{kl3} and denoted $\mathcal{U}_Q(\mathfrak{sl}_m)$ in this 
paper, to a certain $2$-category of colored 
$\mathfrak{sl}_N$-matrix factorizations, denoted $\HMF_{m,d,N}$. 
This $2$-functor is similar 
to Khovanov and Lauda's $2$-functor $\Gamma_d$ from 
$\mathcal{U}_Q(\mathfrak{sl}_m)$ to a $2$-category built from the cohomology 
rings of partial flag varieties (of flags in $\mathbb{C}^d$). However, 
they are not equivalent and do not categorify the same maps, as we 
will explain. 

Now assume that $d=N\ell$, for some $\ell\in\N_{>1}$, and that $m\geq d$. 
Denote by 
$\omega_{\ell}$ the $\ell$-th fundamental $\mathfrak{sl}_m$-weight and 
take $\Lambda=N\omega_{\ell}$. 
We define an additive graded $\mathfrak{sl}_N$-web category 
${\mathcal W}_{\Lambda}^{\circ}$ and use $\Gamma_{m,d,N}$ to define 
a $2$-representation of categorified quantum $\mathfrak{sl}_m$ on it. 

We prove that this implies that $\dot{\mathcal W}_{\Lambda}^{\circ}$, 
the Karoubi envelope of ${\mathcal W}_{\Lambda}^{\circ}$, 
is equivalent to the category 
of finite-dimensional graded projective 
$R_{\Lambda}$-modules, where  
$R_{\Lambda}$ is the level-$N$ cyclotomic KLR-algebra of highest weight 
$\Lambda$. In particular, this implies that the split Grothendieck 
group of $\dot{\mathcal W}^{\circ}_{\Lambda}$ is isomorphic to the corresponding 
web space.   

\vskip0.5cm

The category $\dot{\mathcal W}^{\circ}_{\Lambda}$ decomposes into blocks, 
as we will show. Each of these blocks is equivalent to the category of 
finite-dimensional graded projective modules over a certain finite-dimensional 
algebra, called the $\mathfrak{sl}_N$-{\em web algebra}, 
which is studied in a sequel to the present paper~\cite{mack2}. 

For $N=2$ these algebras were introduced 
by Khovanov~\cite{kh}, who called them {\em arc algebras}. 
Huerfano and Khovanov~\cite{hkh} categorified certain level-two irreducible 
$\mathbf{U}_q(\mathfrak{sl}_m)$-representations using arc algebras.  
In those days categorified quantum groups and 
cyclotomic KLR algebras had not been invented yet, so they did not work out 
the full $\mathfrak{sl}_m$ $2$-representations. But other than that our 
results can be seen as the level-$N$ generalization of theirs. 

The representation theory of the arc algebras and its relation to the 
geometry of 2-block Springer varieties was studied in detail 
in~\cite{bs,bs2,bs3,bs4,bs5,ck,kh,kh2,str,sw}. For $N=3$ the web algebras were 
introduced and studied by Pan, Tubbenhauer and M.M. in~\cite{mpt}. The 
categorified quantum skew Howe duality was proved in that paper. In general 
less is known about the web algebras for $N=3$ than for $N=2$. 
\vskip0.5cm
For $N=2$ and $N=3$ the web category can be defined using cobordisms or 
foams respectively. The reason we do not use foams in this paper, 
is that they have not yet been 
defined for $\mathfrak{sl}_N$ in general. For $N=2$ and $N=3$ it is known that 
the space of $\mathfrak{sl}_N$-foams between two webs is isomorphic to the 
$\mathrm{EXT}$-group of the corresponding matrix factorizations~\cite{kr,mv2}. 
For $N\geq 4$, $\mathfrak{sl}_N$-foams were defined and studied in~\cite{msv2}, but only for the colors 1,2 and 3. To use $\mathfrak{sl}_N$-foams for the 
categorification of $\mathfrak{sl}_N$-webs in general, one would have to 
define $\mathfrak{sl}_N$-foams for all colors and 
find a consistent and complete set of relations on them. Perhaps our 
categorification of quantum skew Howe duality in this paper can help to 
achieve that goal, which in our view would be a proper and complete 
categorification of the results in~\cite{ckm}.  

The results in this paper might also help to find a proof 
that the Khovanov-Rozansky $\mathfrak{sl}_N$-link homology 
is isomorphic to Webster's $\mathfrak{sl}_N$-link 
homology~\cite{we1,we2}, for which he used generalizations of the 
cyclotomic KLR algebras. 
For $N=2$ and $N=3$ a first step in that direction has 
already been taken by Lauda, Queffelec and Rose in~\cite{lqr}. They 
used categorified quantum skew Howe duality 
to prove that Khovanov's (and therefore Khovanov and Rozansky's) 
$\mathfrak{sl}_N$ link homologies are isomorphic to 
link homologies obtained from the so called Chuang-Rouquier complexes 
over level-$N$ cyclotomic KLR algebras, for $N=2,3$.\footnote{The idea to relate Khovanov-Lauda diagrams to foams was first suggested by 
Khovanov to M.~M. in 2008 and worked out in an unpublished 
preprint~\cite{mack} for 
$\mathfrak{sl}_3$ foams over $\Z/2\Z$. In~\cite{mpt} and~\cite{lqr} 
the sign problem in that preprint got fixed; by ``brute force'' in the first 
case and by the introduction of a Blanchet-like version of $\mathfrak{sl}_3$ 
foams in the second case.} 
Their result probably generalizes to arbitrary $N\geq 2$, 
using categorified skew Howe duality and matrix factorizations as 
in this paper. 

Webster~\cite{we2} showed that his link homologies are isomorphic to 
Mazorchuk and Stroppel's representation theoretic link homologies~\cite{ms}, 
and also to the Koszul dual version of these homologies 
independently obtained by Sussan~\cite{su}. Mazorchuk and Stroppel proved that 
their link homologies are isomorphic 
to Khovanov's and Khovanov and Rozansky's link homologies for $N=2,3$, and 
conjectured that to be true for $N\geq 4$ also. Thus we know that all 
the link homologies mentioned in this introduction 
are isomorphic for $N=2,3$. By the remarks above, 
our work might help to prove the same result for arbitrary $N\geq 2$.

We should also remark that there is an algebro-geometric categorification of 
quantum skew Howe duality, due to Cautis, Kamnitzer and Licata~\cite{ckl}. It 
would be interesting to understand the precise relation with the results 
in this paper.   
\vskip0.5cm

The paper is organized as follows:
\begin{itemize}
\item After fixing some notation and conventions in 
Section~\ref{sec:conventions}, we briefly recall 
$\U{n}$ and its fundamental representations in Section~\ref{sec:fund}. 
After this section, we will always use the parameters $m$ and $N$ 
instead of $n$ in order to distinguish the two sides which occur in 
Howe duality.  
\item In Section~\ref{sec:webs} we recall the necessary material 
on $\mathfrak{sl}_N$ webs and quantum skew Howe duality. 
\item In Section~\ref{mf} we recall some material on matrix factorizations, 
which can be found in more detail in~\cite{kr,wu,yo1,yo2}. These 
matrix factorizations categorify colored $\mathfrak{sl}_N$ webs. 
\item In Section~\ref{sec:catgroupandrep} we briefly recall 
Khovanov and Lauda's~\cite{kl3} diagrammatic categorification of 
quantum $\U{m}$, denoted $\u_Q(\mathfrak{sl}_m)$ in this paper, and the 
cyclotomic KLR algebras, denoted $R_{\Lambda}$. We also recall Brundan and 
Kleshchev's~\cite{bk} categorification theorem and Rouquier's~\cite{rou} 
additive universality theorem for $R_{\Lambda}$.  
\item In Sections~\ref{sec:somemorphisms} and~\ref{sec:HMF} 
we prove all the technical results about matrix factorizations 
that are needed in Section~\ref{sec:cathowe}. 
\item In Section~\ref{sec:cathowe} we first give the $2$-functor 
$$\Gamma_{m,d,N}^*\colon \u_Q(\mathfrak{sl}_m)^*\to \HMF_{m,d,N}^*$$ 
and prove that is well-defined in Theorem~\ref{thm:2-funct}. The meaning 
of $*$ is explained in Section~\ref{sec:conventions}.  

Then we define the web categories 
using the matrix factorizations from Section~\ref{mf} and 
prove the aforementioned relation with the level-$N$ cyclotomic KLR algebras 
in Theorem~\ref{thm:categorification1}. As a consequence we 
see that the web categories categorify the web spaces 
in Corollary~\ref{cor:categorification11}. 
\end{itemize}



 
\section{Notation and conventions}\label{sec:conventions}
In this sections fix some notations and explain some conventions. 
\vskip0.5cm
Let $\mathcal{C}^*$ be a $\Z$-graded $\C$-linear additive category 
which admits translation (for a precise definition, see~\cite{l} for 
example). Then $\{t\}$ denotes a positive translation/shift 
of $t$ units. For any 
Laurent polynomial $f(q)=\sum a_iq^i\in \N[q,q^{-1}]$, 
we define 
\begin{equation*}
X^{\oplus f(q)}:=\bigoplus_{i}\left(X\{i\}\right)^{\oplus a_i}.
\end{equation*}

Let $\mathcal{C}$ be the subcategory with the same objects as $\mathcal{C}^*$ 
but only degree-zero morphisms. For any pair of objects 
$X,Y\in \mathcal{C}$, let $\mathrm{Hom}(X,Y)$ 
be the usual hom-space in $\mathcal{C}$. 
Then the hom-space in $\mathcal{C}^*$ is given by  
$$\mathrm{HOM}(X,Y):=\bigoplus_{t\in\Z}\mathrm{Hom}(X\{t\},Y).$$
 
For simplicity, assume that $\mathrm{HOM}(X,Y)$ is finite-dimensional. Then we 
define the {\em $q$-dimension} of $\mathrm{HOM}(X,Y)$ by 
$$\dim_q\,\mathrm{HOM}(X,Y):=\sum_{t\in\Z} q^t\dim\, \mathrm{Hom}(X\{t\},Y).$$
\vskip0.5cm 
Assume that $\mathcal{C}$ is Krull-Schmidt. The split Grothendieck 
group $K_0(\mathcal{C})$ is by definition the Abelian group generated 
by the isomorphism classes of the objects in $\mathcal{C}$ modulo the relation  
$$[X\oplus Y]=[X]+[Y],$$
for any objects $X,Y\in\mathcal{C}$. It becomes a $\Z[q,q^{-1}]$-module, 
by defining 
$$q[X]=[X\{1\}],$$
for any object $X\in\mathcal{C}$. For any 
Laurent polynomial $f(q)=\sum a_iq^i\in \N[q,q^{-1}]$, we get 
$$f(q)[X]=[X^{\oplus f(q)}].$$

Assume that 
$S=\{X_1,\ldots,X_s\}$ is a set of indecomposable objects in 
$\mathcal{C}$ such that 
\begin{itemize}
\item any indecomposable object in $\mathcal{C}$ is 
isomorphic to $X_i\{t\}$ for a certain $i\in\{1,\ldots,s\}$ and $t\in\Z$;
\item for all $i\ne j\in \{1,\ldots,s\}$ and all $t\in\Z$ we have  
$$X_i\not\cong X_j\{t\}.$$
\end{itemize}
Then it is well-known that $K_0(\mathcal{C})$ is freely generated by $S$.  

In this paper we will mostly tensor $K_0(\mathcal{C})$ with $\C(q)$, so 
we define 
$$K_0^q(\mathcal{C}):=K_0(\mathcal{C})\otimes_{\Z[q,q^{-1}]}\C(q).$$
\vskip0.5cm
A $q$-sesquilinear form on $K_0^q(\mathcal{C})$ is by definition a form 
$$\langle \cdot,\cdot\rangle\colon K_0(\mathcal{C})\times K_0(\mathcal{C})\to 
\C(q)$$
satisfying 
\begin{eqnarray*}
\langle f(q)[X],[Y]\rangle &=& f(q^{-1})\langle [X],[Y]\rangle\\
\langle [X],f(q)[Y]\rangle &=& f(q)\langle [X],[Y]\rangle,
\end{eqnarray*}
for any $f(q)\in \C(q)$. There exists a well-known 
$q$-sesquilinear form on $K_0^q(\mathcal{C})$ 
called the {\em Euler form}. It is defined by 
$$\langle [X],[Y]\rangle := \dim_q\, \mathrm{HOM}(X,Y),$$
for any objects $X,Y\in \mathcal{C}$. Note that the Euler form takes 
values in $\N[q,q^{-1}]$ if the $\mathrm{HOM}$-spaces are finite-dimensional.


\section{The special linear quantum algebra and its fundamental representations}
\label{sec:fund}
We briefly recall the special linear quantum algebra and the pivotal 
category of its fundamental representations. 
\subsection{The special linear quantum algebra}
Let $n\in\N_{>1}$ be arbitrary. We write  
$$\alpha_i:=(0,\ldots,1,-1,\ldots,0)\in\Z^{n}$$ with $1$ on the $i$-th 
position, for $i=1,\ldots,n-1$. We denote the Euclidean inner product 
on $\Z^{n}$ by $(\cdot,\cdot)$. 
   
\begin{definition} The {\em quantum special linear algebra} 
${\mathbf U}_q(\mathfrak{sl}_n)$ is 
the associative unital $\C(q)$-algebra generated by $K_i^{\pm 1}, E_{\pm i}$, 
for $i=1,\ldots, n-1$, subject to 
the relations
\begin{gather*}
K_iK_j=K_jK_i,\quad K_iK_i^{-1}=K_i^{-1}K_i=1,
\\
E_{+i}E_{-j} - E_{-j}E_{+i} = \delta_{i,j}\dfrac{K_i-K_i^{-1}}{q-q^{-1}},
\\
K_iE_{\pm j}=q^{\pm (\alpha_i,\alpha_j)}E_{\pm j}K_i,
\\
E_{\pm i}^2E_{\pm j}-(q+q^{-1})E_{\pm i}E_{\pm j}E_{\pm i}+E_{\pm j}E_{\pm i}^2=0,
\qquad\text{if}\quad |i-j|=1,\\
E_{\pm i}E_{\pm j}-E_{\pm j}E_{\pm i}=0,\qquad\text{else}
\end{gather*} 
\end{definition}

Recall that ${\mathbf U}_q(\mathfrak{sl}_n)$ is a Hopf algebra with 
coproduct given by 
\begin{equation*}
\label{eq:coproduct}
\Delta(E_{+i})=E_{+i}\otimes K_i+1\otimes E_{+i},\quad 
\Delta(E_{-i})=E_{-i}\otimes 1+K_i^{-1}\otimes E_{-i},\quad
\Delta(K_i)=K_i\otimes K_i
\end{equation*}
and antipode by 
\begin{equation*}
\label{eq:antipode}
S(E_{+i})=-E_{+i}K_i^{-1},\quad 
S(E_{-i})=-K_iE_{-i},\quad
S(K_i)=K_i^{-1}.  
\end{equation*}
The counit is given by 
\begin{equation*}
\label{eq:counit}
\epsilon(E_{\pm i})=0,\quad \epsilon(K_i)=1.
\end{equation*}
The Hopf algebra structure is used to define 
${\mathbf U}_q(\mathfrak{sl}_n)$ actions on tensor products and duals of 
${\mathbf U}_q(\mathfrak{sl}_n)$-modules. 
\vskip0.5cm
Recall that the ${\mathbf U}_q(\mathfrak{sl}_n)$-weight lattice is 
isomorphic to $\Z^{n-1}$. For any $i=1,\ldots, n-1$, the element $K_i$ acts as 
$q^{\lambda_i}$ on the $\lambda$-weight space of any weight representation. 

Although we have not recalled the definition of 
${\mathbf U}_q(\mathfrak{gl}_n)$, it is sometimes convenient to use 
${\mathbf U}_q(\mathfrak{gl}_n)$-weights. Recall that 
the ${\mathbf U}_q(\mathfrak{gl}_n)$-weight lattice is isomorphic to $\Z^n$ 
and that any $\vec{k}=(k_1,\ldots,k_N)\in\Z^n$ determines a unique 
${\mathbf U}_q(\mathfrak{sl}_n)$-weight 
$$\lambda=(k_1-k_2,\ldots,k_{n-1}-k_n)\in
\Z^{n-1}.$$ 
In this way, we get an isomorphism 
\begin{equation}
\label{eq:isolattices}
\Z^n/\langle(1^n)\rangle\cong \Z^{n-1}.
\end{equation}

Conversely, given a ${\mathbf U}_q(\mathfrak{sl}_n)$-weight 
$\lambda=(\lambda_1,\ldots,\lambda_{n-1})$, there is not a 
unique choice of ${\mathbf U}_q(\mathfrak{gl}_n)$-weight. We first have to 
fix the sum of the entries: for any $d\in\Z$, the equations 
\begin{align}
\label{eq:sl-gl-wts1}
k_i-k_{i+1}&=\lambda_i,\\
\label{eq:sl-gl-wts2}
\qquad \sum_{i=1}^{n} k_i&=d
\end{align}  
determine $\vec{k}=(k_1,\ldots,k_n)$ uniquely, 
if there exists a solution to~\eqref{eq:sl-gl-wts1} 
and~\eqref{eq:sl-gl-wts2} at all. 

\begin{remark}
Since the ${\bf U}_q(\mathfrak{sl}_n)$ and ${\bf U}_q(\mathfrak{gl}_n)$ 
weights and weight lattices are equal those of the 
corresponding classical algebras, we will often refer to them as the 
$\mathfrak{sl}_n$ and $\mathfrak{gl}_n$-weights and weight lattices.  
\end{remark}

For weight representations, we can also use the idempotented 
version of ${\mathbf U}_q(\mathfrak{sl}_n)$, denoted $\U{n}$, due to 
Beilinson, Lusztig and MacPherson~\cite{B-L-M}. 
For $n=2$, define $i'=(2)$. For $n>2$, define    
$$i':=
\begin{cases}
(2,-1,0\ldots,0),&\text{for}\;i=1;\\
(0,\ldots,-1,2,-1,\ldots,0),&\text{for}\; 2\leq i\leq n-2;\\
(0,\ldots,0,-1,2),&\text{for}\;i=n-1.
\end{cases}
$$ 

Adjoin an idempotent $1_{\lambda}$ for each $\lambda\in\Z^{n-1}$ and add 
the relations
\begin{align*}
1_{\lambda}1_{\mu} &= \delta_{\lambda,\mu}1_{\lambda},   
\\
E_{\pm i}1_{\lambda} &= 1_{\lambda \pm i'}E_{i},
\\
K_i1_{\lambda} &= q^{\lambda_i}1_{\lambda}.
\end{align*}
\begin{definition} The {\em idempotented quantum special linear algebra} 
is defined by 
\[
\U{n}=\bigoplus_{\lambda,\mu\in\Z^{n-1}}1_{\lambda}{\mathbf U}_q(\mathfrak{sl}_n)1_{\mu}.
\]
\end{definition} 

The following remark is useful for Proposition~\ref{prop:CKM}.
\begin{remark}
\label{rem:algebraversuscategory}
Sometimes we will consider ${\mathbf U}_q(\mathfrak{sl}_n)$ as a 
$\C(q)$-linear category rather than an algebra. The objects are the 
idempotents $1_{\lambda}$, for $\lambda\in\Z^{n-1}$, and 
$$\mathrm{Hom}(1_{\lambda},1_{\mu})
:=1_{\lambda}{\mathbf U}_q(\mathfrak{sl}_n)1_{\mu}.$$
\end{remark}

\subsection{Fundamental representations}\label{sec:fundreps}
As already remarked in the introduction, from now on we distinguish 
the two sides which occur in Howe duality by using the parameters 
$m$ and $N$ instead of $n$ for the general linear quantum groups.
\vskip0.5cm
In this section we recall the fundamental ${\mathbf U}_q(\mathfrak{sl}_N)$ 
representations, following~\cite{ckm,mor}. 
The basic ${\mathbf U}_q(\mathfrak{sl}_N)$-representation is denoted $\C^N_q$. 
It has a standard basis $\{x_1,\ldots,x_N\}$ on which the action is given by 
$$
E_{+i}(x_j)=
\begin{cases}
x_i,&\text{if}\;j=i+1;\\
0,&\text{else}.
\end{cases}
\quad
E_{-i}(x_j)=
\begin{cases}
x_{i+1},&\text{if}\;j=i;\\
0,&\text{else}.
\end{cases}
$$
$$
K_i(x_j)=
\begin{cases}
qx_i,&\text{if}\;j=i;\\
q^{-1}x_{i+1},&\text{if}\; j=i+1;\\
x_j,&\text{else}.
\end{cases}
$$

Using the basic representation, one can define all fundamental 
${\mathbf U}_q(\mathfrak{sl}_N)$-representations. Define the 
{\em quantum exterior algebra}
$$\Lambda^{\bullet}_q(\C^N_q):=T\C^N_q/\langle \{x_i\otimes x_i,\;x_i\otimes x_j+q
x_j\otimes x_i\mid 1\leq i<j\leq N\}\rangle.$$ 
We denote the equivalence class of $x\otimes y$ by $x \wedge_q y$. 
Note that 
$$\Lambda^{\bullet}_q(\C^N_q)=\bigoplus_{k=0}^N \Lambda^k(\C_q^N)$$
is a graded algebra. For each $0\leq k\leq N$, the homogeneous direct 
summand $\Lambda_q^k(\C_q^N)$ is an irreducible 
${\mathbf U}_q(\mathfrak{sl}_N)$-representation. For $k=0,N$ it is the 
trivial representation and for $1\leq k\leq N$ it is called the 
{\em $k$-th fundamental representation}. It is well-known that the dual of 
the $k$-th fundamental representation is isomorphic to the $(N-k)$-th 
fundamental representation.  
\vskip0.5cm
In this paper we will use tensor products of fundamental representations 
and their duals, which are also 
${\mathbf U}_q(\mathfrak{sl}_N)$-representation by the Hopf 
algebra structure on ${\mathbf U}_q(\mathfrak{sl}_N)$.
\begin{definition}
Let $\mathcal{R}ep(\mathrm{SL}_N)$ be the pivotal category whose objects are 
tensor products of fundamental ${\mathbf U}_q(\mathfrak{sl}_N)$-representations and their duals and 
whose morphisms are intertwiners. 
\end{definition}
Morrison (Theorem 3.5.8 in~\cite{mor}) defined a generating set 
of intertwiners in $\mathcal{R}ep(\mathrm{SL}_N)$, the 
precise definition of which is not 
relevant here. In Section~\ref{sec:webs} we recall 
Cautis, Kamnitzer and Morrison's diagrammatic presentation 
of $\mathcal{R}ep(\mathrm{SL}_N)$ in~\cite{ckm}, which will be used 
in the rest of this paper. 

\section{$\mathrm{SL}_N$ webs}\label{sec:webs}
The morphisms in $\mathcal{R}ep(\mathrm{SL}_N)$ can be represented graphically 
by {\em webs}. These are certain trivalent graphs, whose edges 
are colored by integers belonging to $\{0,\ldots,N\}$. 
Webs can be seen as morphisms in a pivotal category, which in the literature is 
called a {\em spider} or {\em spider category}, denoted 
$\mathcal{S}p(\mathrm{SL}_{N})$.  
\subsection{The $\mathrm{SL}_N$ spider}
\label{sec:spider}

Recently, Cautis, Kamnitzer and Morrison~\cite{ckm} gave a presentation 
of $\mathcal{S}p(\mathrm{SL}_{N})$ in terms of generating webs and relations. 

\begin{definition}[Cautis-Kamnitzer-Morrison]
\label{def:webrelations}
The objects of $\mathcal{S}p(\mathrm{SL}_N)$ are finite sequences 
$\vec{k}$ of elements in $\{0^{\pm},\ldots,(N)^{\pm}\}$. 

The hom-space $\mathrm{Hom}(\vec{k},\vec{l})$ is the 
$\mathbb{C}(q)$ vector space freely generated by all diagrams, with 
lower and top boundary labeled by $\vec{k}$ and $\vec{l}$ 
respectively, which can be obtained by glueing and juxtaposing 
labeled cups and caps and the following elementary webs, together with 
the ones obtained by mirror reflections and arrow reversals: 

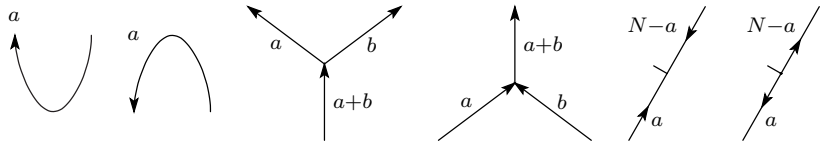
\begin{eqnarray*}
\txt{\input{figure/diag26}}
\hspace{0.5cm}
\txt{\input{figure/diag27}}
\hspace{0.5cm}
\txt{\input{figure/diag1}}
\hspace{0.5cm}
\txt{\input{figure/diag2}}
\hspace{0.5cm}
\txt{\input{figure/diag4}}
\hspace{0.5cm}
\txt{\input{figure/diag3}}
\end{eqnarray*}
with all labels between $0$ and $N$,  
modded out by planar isotopies (i.e. zig-zag relations for cups and caps) 
and the following relations:
\begin{eqnarray}
\label{eq:tagswitch}\txt{\input{figure/diag4}}&\hspace{0.7cm}=&(-1)^{a(N-a)}\txt{\input{figure/diag5}}
\\[1em]
\label{eq:paralleldigon}\txt{\input{figure/diag12}}&\hspace{0.7cm}=&\begin{bmatrix}a+b\\a\end{bmatrix}_q\txt{\unitlength 0.1in
\begin{picture}(  1.0200,  6.0000)(  3.7300, -8.0000)
\special{pn 8}%
\special{pa 400 800}%
\special{pa 400 200}%
\special{fp}%
\special{sh 1}%
\special{pa 400 500}%
\special{pa 380 568}%
\special{pa 400 554}%
\special{pa 420 568}%
\special{pa 400 500}%
\special{fp}%
\put(6.0000,-5.0000){\makebox(0,0){${}_{b+a}$}}%
\end{picture}%
}
\\[1.5em]
\label{eq:oppositedigon}\txt{\input{figure/diag13}}&\hspace{0.7cm}=&\begin{bmatrix}N-a\\b\end{bmatrix}_q\txt{\unitlength 0.1in
\begin{picture}(  1.0200,  6.0000)(  3.7300, -8.0000)
\special{pn 8}%
\special{pa 400 800}%
\special{pa 400 200}%
\special{fp}%
\special{sh 1}%
\special{pa 400 500}%
\special{pa 380 568}%
\special{pa 400 554}%
\special{pa 420 568}%
\special{pa 400 500}%
\special{fp}%
\put(5.50000,-5.0000){\makebox(0,0){${}_{a}$}}%
\end{picture}%
}
\\[1em]
\label{eq:associativity}\txt{\input{figure/diag14}}&\hspace{0.7cm}=&\txt{\input{figure/diag15}}
\\[1em]
\label{eq:parallelsquare}\hspace{1cm}\txt{\input{figure/diag17}}\hspace{0.5cm}&\hspace{0.7cm}=&\begin{bmatrix}s+t\\t\end{bmatrix}_q\hspace{1cm}\txt{\input{figure/diag16}}
\\[1em]
\label{eq:oppositesquare}\hspace{1cm}\txt{\input{figure/diag18}}\hspace{0.5cm}&\hspace{0.7cm}=&\sum_r \begin{bmatrix}a-b+t-s\\r\end{bmatrix}_q\hspace{1cm}\txt{\input{figure/diag19}}
\end{eqnarray}
together with the analogous relations obtained by mirror reflections and 
arrow reversals. 
\end{definition}
Note that with tags one can invert the orientation of cups and caps, so one 
only needs the above cups and caps as generators.  
\vskip0.5cm
The following result can be found in~\cite{ckm} (Theorems 3.2.1 and 3.3.1):
\begin{theorem}[Cautis-Kamnitzer-Morrison]\label{thm:ckm}
There exists an equivalence of pivotal categories $\gamma_N\colon 
\mathcal{S}p(\mathrm{SL}_N)\to \mathcal{R}ep(\mathrm{SL}_N)$, 
which on objects is defined by  
\begin{eqnarray*}
\vec{k}=(k_1^{\epsilon_1},\ldots,k_m^{\epsilon_m})&\mapsto& 
\Lambda^{\vec{k}}_q(\C^N_q)=(\Lambda^{k_i}_q(\C^N_q))^{\epsilon_1}\otimes 
\cdots \otimes (\Lambda^{k_m}_q(\C^N_q))^{\epsilon_m},\\
\end{eqnarray*}  
where $V^1:=V$ and $V^{-1}:=V^*$. 
\end{theorem}

\begin{remark}
There is a slight discrepancy with the setup in~\cite{ckm}. We allow the 
colors $0$ and $N$ too, whereas Cautis, Kamnitzer and Morrison do not.  
Of course $\Lambda^0(\C_q^N)$ and $\Lambda^N(\C_q^N)$ are both isomorphic to 
the trivial ${\mathbf U}_q(\mathfrak{sl}_N)$-representation and one can 
decide to not draw edges labeled with them. However, for our purposes they 
are actually useful.    
\end{remark}


\subsection{Quantum skew Howe duality}
\label{sec:Howe}
Let us briefly recall the instance of quantum skew Howe duality 
which we will categorify. For more details, see~\cite{ckm} and the 
references therein. 
\vskip0.5cm
{\em For the rest of this paper, let $N\geq 2 $ and $m,d\geq 0$ be 
arbitrary integers.} 
\vskip0.5cm
Below we are only interested in $\mathfrak{gl}_m$-weights with 
entries between $0$ and $N$. Following~\cite{ckm}, we call these 
weights {\em $N$-bounded}. Let $\Lambda(m,d)_N$ denote the set of 
$N$-bounded $\mathfrak{gl}_m$-weights whose entries sum up to $d$, i.e. 
$$\Lambda(m,d)_N:=
\{\vec{k}\in \{0,\ldots,N\}^m\mid k_1+\cdots+k_m=d\}.$$
\begin{definition}
\label{def:slnglnweights} 
We define the map 
$$\phi_{m,d,N}\colon \Z^{m-1}\to \Lambda(m,d)_N\cup \{*\}$$ by 
\[
\phi_{m,d,N}(\lambda)=\begin{cases}
\vec{k},& \text{if~\eqref{eq:sl-gl-wts1} and \eqref{eq:sl-gl-wts2} 
have a solution in $\Lambda(m,d)_N$};\\
*,& \text{otherwise}.
\end{cases}
\]   
\end{definition}

The following Proposition is due to Cautis, Kamnitzer and Morrison and 
follows from Propositions 5.1.2 and 5.2.1 in~\cite{ckm}. In this proposition 
we consider $\U{m}$ as a $\C(q)$-category rather than an idempotented 
algebra (see Remark~\ref{rem:algebraversuscategory}). 
\begin{proposition}[Cautis-Kamnitzer-Morrison]
\label{prop:CKM}
The functor $$\gamma_{m,d,N}\colon 
\Uv{m}\to \mathcal{S}p(\mathrm{SL}_N)$$ determined by 
\begin{eqnarray*}
1_{\lambda}&\mapsto& 
\begin{cases}
\hspace{1cm}\txt{\input{figure/id-lambda}},&\hspace{1cm}\text{if}\;
\phi_{m,d,N}(\lambda)=\vec{k};\\
0,&\hspace{1cm}\text{if}\;\phi_{m,N,d}(\lambda)=*.
\end{cases}
\\[0.5em]
E_{+i}1_{\lambda}&\mapsto& \txt{\input{figure/e-mu-i-1}}\\
E_{-i}1_{\lambda}&\mapsto& \txt{\input{figure/f-mu-i-1}}\\
\end{eqnarray*} 
is well-defined.  
\end{proposition}

\begin{remark}
We label the vertical edges with the entries of $\vec{k}$ in 
the reverse order. This differs from the convention in~\cite{ckm}, but 
their results hold true with our convention too. The reason that we opted for 
this ``opposite'' convention, is that it permits us very easily to compare 
our $2$-representation $\Gamma_{m,N}$ in Theorem~\ref{thm:2-funct} to the 
$2$-representations in~\cite{kl3} and~\cite{msv}. Finding a consistent 
set of signs for the definition of such $2$-representations is very tricky, 
so we prefer to stick to the conventions in those two papers.  
\end{remark}

Note that by~\eqref{eq:paralleldigon} and~\eqref{eq:associativity}, 
it is easy to determine the images of the divided powers 
$$E_{+i}^{(a)}:=E_{+i}^a/[a]!\quad\text{and}\quad E_{-i}^{(a)}=E_{-i}^a/[a]!$$
\begin{eqnarray*}
E_{+i}^{(a)}1_{\lambda}&\mapsto& \txt{\input{figure/e-mu-i-a}}\\
E_{-i}^{(a)}1_{\lambda}&\mapsto& \txt{\input{figure/f-mu-i-a}}\\
\end{eqnarray*} 
\vskip0.5cm
Proposition~\ref{prop:CKM} singles out a special class of web diagrams, 
called {\em ladders}, which Cautis, Kamnitzer and Morrison defined in 
their Section 5. 
\begin{definition}[Cautis-Kamnitzer-Morrison]
\label{def:ladders}
An {\em $N$-ladder with $m$ uprights} is a rectangular $\mathfrak{sl}_N$-web 
diagram without tags,  
\begin{itemize}
\item whose vertical edges are all oriented upwards and lie on $m$ 
parallel vertical lines running from bottom to top; 
\item which contains a certain number of horizontal oriented {\em rungs} 
connecting adjacent uprights.
\end{itemize} 
\end{definition}
\noindent Since ladders do not have tags, at each trivalent vertex 
the sum of the labels of the incoming edges has to be equal to the sum of 
the labels of the outgoing edges. 

It is clear that any $N$-ladder with $m$ uprights whose labels add up to $d$, 
at any generic level between the rungs, is the image under 
$\gamma_{m,d,N}$ of a product of divided powers in $\U{m}$ by 
Proposition~\ref{prop:CKM}. 
\begin{remark}
In~\cite{ckm} the authors are very careful in distinguishing between a 
ladder, which for them is just a diagram in the ``free spider'', and its 
image in $\mathcal{S}p(\mathrm{SL}_N)$. In this paper that distinction is 
irrelevant and we use the term ``ladder'' more loosely to designate any 
web which can be represented by a ladder diagram in 
$\mathcal{S}p(\mathrm{SL}_N)$.
\end{remark}
\vskip0.5cm
Now suppose that $d=N\ell$, for some integer $\ell\geq 0$, and 
that $m\geq d$. Let $\Lambda=N\omega_{\ell}$ be 
$N$ times the $\ell$-th fundamental 
$\mathfrak{sl}_m$-weight. We denote by $P_{\Lambda}$ 
the set of $\mathfrak{sl}_m$-weights in the irreducible 
$\Uv{m}$-module $V_{\Lambda}$. Note that $\phi_{m,d,N}(P_{\Lambda})=\Lambda(m,d)_N$. 
In particular, we have 
$\phi_{m,d,N}(\Lambda)=(N^{\ell})\in \Lambda(m,d)_N$. 

\begin{definition}\label{def:webmodule}
Define the $\Uv{m}$-{\em web module} with highest weight $\Lambda$ by
$$W_{\Lambda}:=\bigoplus_{\vec{k}\in \Lambda(m,d)_N}W(\vec{k},N),$$
where $W(\vec{k},N)$ is the {\em web space} defined by  
$$W(\vec{k},N):=\mathrm{Hom}((N^{\ell}), \vec{k}).$$   
\end{definition}

Let $\vec{k},\vec{k}'\in \Lambda(m,d)_N$. Any web in 
$\mathrm{Hom}(\vec{k},\vec{k}')$ defines a linear map 
$W(\vec{k},N)\to W(\vec{k}',N)$ by glueing it on top of the 
webs in $W(\vec{k},N)$. Therefore, the homomorphism $\gamma_{m,d,N}$ 
induces a well-defined action of $\U{m}$ on $W_{\Lambda}$. We are going 
to show that $W_{\Lambda}$ is an irreducible $\U{m}$-representation. 
\vskip0.5cm
For any $u\in W(\vec{k},N)$, 
let $$u^*\in \mathrm{Hom}(\vec{k},(N^{\ell}))$$ be the web obtained via 
reflexion in the $x$-axis and reorientation. Note that $u$ and $u^*$ can 
be glued together such that $u^*u\in \mathrm{End}((N^{\ell}))$. 

Let 
$$\mathrm{ev}\colon \mathrm{End}((N^{\ell}))\to\C(q)$$ 
denote the 
isomorphism given by the evaluation of closed webs (forgetting about the 
edges labeled $N$). We can define a 
$q$-sesquilinear form 
on $W(\vec{k},N)$, which we call the {\em $q$-sesquilinear web form}, by
\begin{equation}
\label{eq:webbracket}
\langle u, v\rangle:= q^{d(\vec{k})}\mathrm{ev}(u^*v)\in \C(q).
\end{equation}
for any pair of monomial webs $u,v\in W(\vec{k},N)$. 
The normalization factor is defined by  
$$d(\vec{k})=1/2\left(N(N-1)\ell-\sum_{i=1}^{m}k_i(k_i-1)\right).$$
We assume that the web form is $q$-antilinear in the first entry and 
$q$-linear in the second. 
\vskip0.5cm

Before we state the following result, we recall 
the $q$-antilinear algebra anti-involution 
$\tau$ on $\Uv{m}$ which is defined by 
\[
\tau(1_{\lambda})=1_{\lambda},\;\;
\tau(1_{\lambda+i'}E_{+i}1_{\lambda})= 
q^{-1-\lambda_i}1_{\lambda}E_{-i}1_{\lambda+i'},\]
\[
\tau(1_{\lambda}E_{-i}1_{\lambda+i'})= q^{1+\lambda_i}
1_{\lambda+i'}E_{+i}1_{\lambda}.
\] 
Let $V_{\Lambda}$ be the irreducible $\Uv{m}$-representation with highest 
weight $\Lambda$, which is unique up to isomorphism. 
The $q$-{\em Shapovalov form} $\langle\cdot,\cdot\rangle$ on 
$V_{\Lambda}$ is the unique $q$-sesquilinear form such that 
\begin{enumerate}
\item $\langle v_{\Lambda},v_{\Lambda}\rangle =1$, for a fixed 
highest weight vector $v_{\Lambda}$;
\item $\langle x v, v' \rangle=\langle v,\tau(x) v'\rangle$, for any 
$x\in \U{m}$ and any $v,v'\in V_{\Lambda}$.
\end{enumerate}  
We recall that this form is non-degenerate.

\begin{corollary}
\label{cor:webirrep}
$W_{\Lambda}$ is an irreducible $\Uv{m}$ 
representation with highest weight $\Lambda$. The $q$-sesquilinear 
web form is equal to the $q$-Schapovalov form. 
\end{corollary}
\begin{proof}
Note that $E_{+i}W(\Lambda,N)=\{0\}$ for 
any $i=1,\ldots, m-1$, so the web of $\ell$ parallel vertical upward oriented 
$N$-edges, denoted $w_{\Lambda}$, is a highest weight vector of weight $\Lambda$. 

By Theorem 3.3.1 in~\cite{ckm}, we see that 
$$\dim\; W(\vec{k},N)=\dim\; \mathrm{Inv}(V(\vec{k},N))$$
for any $\vec{k}\in \phi_{m,d,N}(P_{\Lambda})$. By Theorem 4.2.1(3) in the same 
paper, this implies that 
$$\dim\; W_{\Lambda}=\dim\; V_{\Lambda}.$$ 

The first statement in the corollary now follows by the uniqueness up to 
isomorphism of irreducible highest weight $\U{m}$-modules.   
\vskip0.5cm
The web form clearly satisfies $\langle w_{\Lambda},w_{\Lambda}\rangle=1$. Note 
that from the definition of the action it follows immediately that  
$$(E_{+i}1_{\lambda}u)^*v=u^*(1_{\lambda}E_{-i}v),$$
for any $i=1,\ldots,N-1$ and any $\lambda\in \Z^{N-1}$. 
Let us assume that $\phi_{m,d,N}(\lambda)=\vec{k}$ (otherwise there is nothing to 
prove) and $\phi_{m,d,N}(\lambda+i')=\vec{k}'$. Then 
$$
\vec{k}'=(k_1,\ldots,k_i+1,k_i-1,\ldots,k_m),
$$
so 
$$
\sum_{j=1}^m k_j'(k'_j-1)=\sum_{j=1}^m k_j(k_j-1)+2
+2(k_i-k_{i+1}),
$$
which implies 
$$
d(\vec{k}')=d(\vec{k})-1-(k_i-k_{i+1}).
$$
The normalization of the web form then gives
$$\langle E_{+i}1_{\lambda}u,v\rangle=q^{d(\vec{k}')}
\mathrm{ev}((E_{+i}1_{\lambda}u)^*v)$$
and 
$$\langle u,\tau(E_{+i}1_{\lambda})v\rangle=q^{-1-(k_i-k_{i+1})}
\langle u,1_{\lambda}E_{-i}v\rangle=$$
$$q^{d(\vec{k})-1-(k_i-k_{i+1})}\mathrm{ev}(u^*(1_{\lambda}E_{-i}v)).$$
This shows that 
$$\langle E_{+i}1_{\lambda}u,v\rangle=\langle u,\tau(E_{+i}1_{\lambda})v\rangle.$$
Similarly, one can check that 
$$\langle E_{-i}1_{\lambda}u,v\rangle=\langle u,\tau(E_{-i}1_{\lambda})v\rangle.$$
Since any element in $\Uv{m}$ is a linear combination of products of 
$E_{+i}$'s and $E_{-i}$'s, this shows that the web form satisfies the defining 
axioms of the $q$-Shapovalov form. 
\end{proof}

Corollary~\ref{cor:webirrep} shows that any monomial web in $W_{\Lambda}$ 
is equal to a linear combination of $N$-ladders with $m$ uprights. The 
following corollary shows that ladders are special. 

\begin{corollary}
\label{cor:positivityforclosedwebs}
Let $u\in W(\Lambda,N)$ be any $N$-ladder with $m$ uprights. Then there 
exists an $\alpha\in\N[q,q^{-1}]$ such that 
$$u=\alpha w_{\Lambda}.$$
\end{corollary}
\begin{proof}
Since $u$ is a ladder, Proposition~\ref{prop:CKM} shows that it 
can be obtained from $w_{\Lambda}$ by applying a product of 
divided powers of $E_{\pm i}$'s. 
The result now follows from Theorem 
22.1.7 in~\cite{lu} and the fact that $w_{\Lambda}$ is the unique 
canonical basis element of weight $\Lambda$.
\end{proof}

\begin{remark}
A good question is whether any monomial web $w\in W(\vec{k},N)$ can be 
represented as an $N$-ladder with $m$ uprights. For $N=2$ and $N=3$ 
the answer is affirmative. For $N=2$ this is immediate and for $N=3$ 
this was proved in Lemma 5.2.10 in~\cite{mpt}. For $N>3$ we do not know the 
answer. Theorem 3.5.1 in~\cite{ckm} cannot be applied, 
at least not immediately, because it 
only shows that $w$ is isomorphic to the image of an $N$-ladder with 
a ``certain number'' of uprights after removing the edges labeled $0$ and $N$.  
The algorithm in the proof of that theorem will give a ladder with 
more than $m$ uprights in general.  
\end{remark}
%
%
%
%
\section{The categorification of webs}
\label{mf}
Wu~\cite{wu} and independently Y.~Y~\cite{yo1,yo2} defined matrix 
factorizations associated to colored 
$\mathfrak{sl}_N$ webs, generalizing the ground breaking work of 
Khovanov and Rozanksy~\cite{kr}. In this section we recall 
these matrix factorizations, which we will use to define our web categories 
in Definition~\ref{def:web-category}.
\subsection{Partially symmetric polynomials}
\indent
Let $k$ and $r$ be two integers. We define  
$$\mathbb{T}^k_r:=\{t_{1,r},...,t_{|k|,r}\},$$
where the $t_{i,r}$ are certain formal variables, which are graded by putting 
$\deg(t_{i,r})=2$, for all $i=1,\ldots,k$. The integer $r$ just serves as 
an index, which will be useful later. 
 
Denote the elementary symmetric and the complete 
symmetric functions in $\mathbb{T}^k_r$ by 
$$e_{0,r}=1,e_{1,r},...,e_{k,r}\quad\text{and}\quad 
h_{0,r}=1,h_{1,r},\ldots,h_{k,r},\ldots$$
respectively. 

Write
$$
s(k)=\left\{ \begin{array}{cc} 1&k\geq 0\\-1&k<0 \end{array} \right. .
$$
Let $m$ be a non-negative integer and 
let $\vec{k}=(k_1,k_2...,k_m)$ and $\vec{r}=(r_1,\ldots,r_m)$ be $m$-tuples of 
integers and define $|\vec{k}|:=\sum_{i=a}^m |k_a|$. We write 
$$
\mathbb{T}^{\vec{k}}_{\vec{r}}:=\mathbb{T}^{k_1}_{r_1}\cup\cdots\cup 
\mathbb{T}^{k_m}_{r_m}.$$
Let 
$$R^{\vec{k}}_{\vec{r}}:=\mathbb{S}(\mathbb{T}^{k_1}_{r_1}
\vert\mathbb{T}^{k_2}_{r_2}\vert \cdots
\vert\mathbb{T}^{k_m}_{r_m}),$$
which is the ring of partially symmetric polynomials which are symmetric in 
each $\mathbb{T}^{k_a}_{r_a}$ separately for $a=1,\ldots,m$. 

We define the rational function $X^{\vec{k}}_{\vec{r}}$ in the alphabet 
$\mathbb{T}^{\vec{k}}_{\vec{r}}$ by 
$$
X^{\vec{k}}_{\vec{r}}:=\prod_{a=0}^{m}\left(\sum_{b=0}^{|k_{a}|}e_{b,r_{a}}\right)^{s(k_{a})}.
$$
Expand $X^{\vec{k}}_{\vec{r}}$ as a power series and let 
$X_{j,\vec{r}}^{\vec{k}}$ ($j\in\Z_{\geq 0}$) be the polynomial given by the 
homogeneous summand of $X^{\vec{k}}_{\vec{r}}$ of degree $2j$. Note that 
$$X_{j,\vec{r}}^{\vec{k}}\in R^{\vec{k}}_{\vec{r}}$$
for all $j\in \Z_{\geq 0}$. 
\begin{example}
Let $m=1$ and $k\geq 0$ ($r$ is just an index, which is not important here), 
then 
$$X_{j,r}^k=e_{j,r}\quad\text{and}\quad X_{j,r}^{-k}=(-1)^jh_{j,r}$$ 
for $j=1,\ldots,k$.  
\end{example}
Define  
$$\mathbb{X}^{\vec{k}}_{\vec{r}}:=\left\{X_{j,\vec{r}}^{\vec{k}}\mid 
0\leq j\leq |\vec{k}|\right\},$$ 
which we treat as an alphabet in its own right. 
\vskip0.5cm
Fix $N\geq 2$ and let 
$$p_{N+1}(\mathbb{T}^k_r):=t_{1,r}^{N+1}+t_{2,r}^{N+1}+...+t_{k,r}^{N+1}.$$
Define $P_{N+1}(\mathbb{X}^{\vec{k}}_{\vec{r}})$ by 
$$
P_{N+1}(\mathbb{X}^{\vec{k}}_{\vec{r}}):=
P_{N+1}(\mathbb{X}^{k_1}_{r_1})+P_{N+1}(\mathbb{X}^{k_2}_{r_2})+...+
P_{N+1}(\mathbb{X}^{k_{m}}_{r_m})
$$
where 
$$
P_{N+1}(\mathbb{X}^{k_a}_{r_a}):=p_{N+1}(\mathbb{T}^{k_a}_{r_a}) 
$$
for each $a=1,\ldots,m$. 
\begin{example}
Let $m=1$ and $k,N=2$. Then 
$$p_{3}(t_1,t_2)=t_1^3+t_2^3=(t_1+t_2)^3-
3(t_1+t_2)t_1t_2,$$ 
so 
$$P_3(e_1,e_2)=e_1^3-3e_1e_2.$$
\end{example}

%
%
%
%
\subsection{Graded matrix factorizations}
Let $R=\C[X_1,...,X_k]$ and put a grading on $R$ by taking $\deg(X_i)$ to 
be an even positive integer for each $i=1,\ldots,k$. Let $P$ be a 
polynomial in $R$.

\begin{definition}
\label{def:matrixfactor}
A {\em graded matrix factorization with potential $P$} is a 
4-tuple 
$$\widehat{M}=(M_0,M_1,d_{M_0},d_{M_1})$$ 
such that 
\begin{equation*}
\xymatrix{M_0\ar[r]^{d_{M_0}}&M_1\ar[r]^{d_{M_1}}&M_0},
\end{equation*}
is a $2$-chain of free graded $R$-modules (possibly of infinite rank) such 
that  
$$\deg(d_{M_0})=\deg(d_{M_1})=\frac{1}{2}\,\deg (P)$$ 
and  
$$d_{M_1} d_{M_0} = P\,\id_{M_0}\quad\text{\and}\quad 
d_{M_0} d_{M_1} = P\,\id_{M_1}.$$
\end{definition}

\begin{definition}
A matrix factorization $\widehat{M}=(M_0,M_1,d_{M_0},d_{M_1})$ is finite if as an $R$-module $\rank(M_0)\,(=\rank(M_1))<\infty$. 
\end{definition}

We define a {\em grading shift} $\{ m \}$ ($m \in \Z$) and a {\em translation} 
$\langle 1\rangle$ on $\widehat{M}=(M_0,M_1,d_{M_0},d_{M_1})$ by
\begin{eqnarray*}
\widehat{M}\{m\}&=&(M_0\{ m \},M_1\{ m\},d_{M_0},d_{M_1})\\
\widehat{M}\langle 1\rangle&=&(M_1,M_0,-d_{M_1},-d_{M_0}).
\end{eqnarray*}
\indent
A {\em morphism $f:\widehat{M}\longrightarrow\widehat{N}$ of matrix 
factorizations} 
is a pair of degree preserving $R$-module morphisms 
$f_0:M_0\to N_0$ and $f_1:M_1\to N_1$ such that
\begin{eqnarray}
\nonumber
d_{N_0}f_0=f_1d_{M_0},\quad d_{N_1}f_1=f_0d_{M_1}.
\end{eqnarray}
A morphism $f:\widehat{M}\longrightarrow\widehat{N}$ of matrix 
factorizations is {\em null-homotopic} if there exists a pair of 
$R$-module morphisms $h_0:M_0\to N_1$ and $h_1:M_1\to N_0$ such that
\begin{eqnarray}
\nonumber
f_0=d_{N_1}h_0+h_1d_{M_0},\quad f_1=d_{N_0}h_1+h_0d_{M_1}.
\end{eqnarray}
Two such morphisms $f,g$ are {\em homotopic} if $f-g$ is 
null-homotopic. 

Let $\HMF_R(P)$ be the homotopy category of matrix factorizations with 
potential $P$. This is an additive Krull-Schmidt category (see Propositions 
24 and 25 in~\cite{kr}). Recall that $\HMF_R^*(P)$, which was defined in 
Section~\ref{sec:conventions}, contains all homogeneous morphisms. 
\vskip0.5cm

Let $\mathbb{X}$ and $\mathbb{Y}$ be two sets of variables and put 
$\mathbb{U}=\mathbb{X}\cap \mathbb{Y}$ and 
$\mathbb{V}=\mathbb{X}\cup\mathbb{Y}$.
Take $R=\C[\mathbb{X}]$, $R'=\C[\mathbb{Y}]$ and $S=\C[\mathbb{U}]$ and 
$Q=\C[\mathbb{V}]$. Note that 
$$Q=R\otimes_S R'.$$
For $\widehat{M}=(M_0,M_1,d_{M_0},d_{M_1})$ in $\HMF_{R}(P)$ and 
$\widehat{N}=(N_0,N_1,d_{N_0},d_{N_1})$ in $\HMF_{R'}(P')$, 
we define the tensor product 
$\widehat{M} \btime{S}\widehat{N}$ in $\HMF_{Q}(P+P')$ 
by 
\begin{gather*}
\widehat{M} \btime{S}\widehat{N}:=\\
\left(
\left(\begin{array}{c}
		M_0\ostimes N_0\\
		M_1\ostimes N_1
	\end{array}
\right),
\left(\begin{array}{c}
		M_1\ostimes N_0\\
		M_0\ostimes N_1
	\end{array}
\right),
\left(
	\begin{array}{cc}
		d_{M_0}&-d_{N_1}\\
		d_{N_0}&d_{M_1}
	\end{array}
\right),
\left(
	\begin{array}{cc}
		d_{M_1}&d_{N_1}\\
		-d_{N_0}&d_{M_0}
	\end{array}
\right)\right)
.
\end{gather*}
If no confusion is possible, we will write 
$\widehat{M} \boxtimes\widehat{N}$. 
\begin{example}
Let $p$ and $q$ be two homogeneous polynomials in a graded polynomial ring 
$R$ and let $M$ be a free graded $R$-module. We define the matrix factorization 
$K(p;q)_{M}$ with potential $p q$ by
\begin{eqnarray}
\nonumber
K(p;q)_{M}&:=& 
(M,M\{ \frac{1}{2}(\, \deg (q)-\deg (p)\, )\},p,q)
.
\end{eqnarray}
\indent
More generally, for sequences $\mathbf{p}=(p_1, p_2, ..., p_r)$, 
$\mathbf{q}=(q_1, q_2, ..., q_r)$ of homogeneous polynomials in $R$, 
we define the matrix factorization 
$K\left( \mathbf{p} ; \mathbf{q} \right)_{M}$ with 
potential $\sum_{i=1}^r p_i q_i$ by 
\begin{eqnarray*}
K\left( \mathbf{p} ; \mathbf{q} \right)_{M}
&:=&
\btime{\,\,R}_{i=1}^r K(p_i;q_i)_{R}\btime{R}(M,0,0,0)
.
\end{eqnarray*}
These matrix factorizations are called 
{\em Koszul matrix factorizations}~\cite{kr}.
\end{example}
%
%
%

\subsection{The $2$-complex $\HOM_R(\widehat{M},\widehat{N})$}
We define the structure of a $2$-complex on $\HOM_R(\widehat{M},\widehat{N})$ 
by
\begin{equation}
\nonumber
\xymatrix{
\HOM_R^0(\widehat{M},\widehat{N})\ar[r]^{d_0}&\HOM_R^1(\widehat{M},\widehat{N})\ar[r]^{d_1}&\HOM_R^0(\widehat{M},\widehat{N}),
}
\end{equation}
where
\begin{eqnarray}
\nonumber
\HOM_R^0(\widehat{M},\widehat{N})=\HOM_R(M_0,N_0)\oplus\HOM_R(M_1,N_1),\\
\nonumber
\HOM_R^1(\widehat{M},\widehat{N})=\HOM_R(M_0,N_1)\oplus\HOM_R(M_0,N_1),
\end{eqnarray}
and
\begin{equation}
\nonumber
d_i(f)=d_N\, f+(-1)^if\,d_M\quad (i=0,1).
\end{equation}
The cohomology of this complex is denoted by
\begin{equation}
\nonumber
\EXT(\widehat{M},\widehat{N})=\EXT^0(\widehat{M},\widehat{N})\oplus\EXT^1(\widehat{M},\widehat{N}).
\end{equation}
By definition, we have the following proposition.
\begin{proposition}\label{ext-hom}
We have
\begin{eqnarray*}
\EXT^0(\widehat{M},\widehat{N})&\simeq&\HOM_{\HMF}(\widehat{M},\widehat{N}),\\
\EXT^1(\widehat{M},\widehat{N})&\simeq&\HOM_{\HMF}(\widehat{M},\widehat{N}\langle 1\rangle).
\end{eqnarray*}
\end{proposition}
We also recall the following result, which can be found in Proposition 
12 and Corollary 6~\cite{kr}. Given a matrix factorization 
$\widehat{N}=(N_0,N_1,d_{N_0},d_{N_1})$, one can define its dual by  
$\widehat{N}_{\bullet}=(N_0^{\ast},N_1^\ast,-d_{N_1}^\ast,d_{N_0}^\ast)$, where $N^\ast=\HOM_R(N,R)$.

\begin{lemma}
\label{lem:ExtvH}
If $M$ is finite, we have an isomorphism 
$$
\EXT(\widehat{M},\widehat{N})\cong H(\widehat{M}_{\bullet}\btime{R}\widehat{N})
$$
which preserves the $q$-degree. 
\end{lemma}

%
%
%
%
\subsection{Matrix factorizations associated to webs}\label{sec:MFwebs}
For a given web $\Gamma$, we will always denote the corresponding 
matrix factorization by $\widehat{\Gamma}$. 
\vskip0.5cm

We define a matrix factorization for the following web with formal 
indices associated to its boundaries. We assume that $k\geq 1$.
\begin{figure}[htb]
$$
\txt{\input{figure/figmoy3-mf}}
$$
\caption{${L}^{[k]}_{(1;2)}$}
\end{figure}

\begin{definition}
We define the matrix factorization
\begin{equation}
\label{line-mf}
\widehat{L}^{[k]}_{(1;2)}:=\mathop{\boxtimes}_{a=1}^{k} K\Big( P^{[k]}_{a,(1;2)} ;X_{a,(1)}^{(k)}-X_{a,(2)}^{(k)} \Big)_{R_{(1,2)}^{(k,k)}} 
\end{equation}
where 
\begin{gather*}
P^{[k]}_{a,(1;2)}= \\[0.5em]
\dfrac{P_{N+1}(X_{1,(2)}^{(k)},...,X_{a-1,(2)}^{(k)} ,
X_{a,(1)}^{(k)},...,X_{k,(1)}^{(k)})-P_{N+1}
(X_{1,(2)}^{(k)},...,X_{a,(2)}^{(k)},X_{a+1,(1)}^{(k)},...,X_{k,(1)}^{(k)})}
{X_{a,(1)}^{(k)}-X_{a,(2)}^{(k)}}.
\end{gather*}
\end{definition}

\begin{proposition}\label{str-mf-ind1}
We have the following results:
\begin{enumerate}
\item $\widehat{L}^{[k]}_{(1;2)}$ is homotopy equivalent to the zero matrix 
factorization if $k\geq N+1$;
\vskip0.2cm
\item $\widehat{L}^{[k]}_{(1;2)}$ ($1\leq k\leq N$) is indecomposable.
\end{enumerate}
\end{proposition}
\begin{proof}
(1) Expressing the $N+1$-th power sum in terms of the elementary symmetric 
polynomials, we see that 
$$P_{N+1,(1;2)}^{[k]}\in\C$$ 
Therefore, $K(P_{N+1,(1;2)}^{[k]};X_{N,(1)}^{(k)}-X_{N,(2)}^{(k)})$ is homotopy 
equivalent to the zero matrix factorization, which implies that 
$\widehat{L}^{[k]}_{(1;2)}$ is homotopy equivalent to the zero matrix 
factorization.\\ 
(2) We have 
$$\EXT(\widehat{L}_{(1;2)}^{[k]}, \widehat{L}_{(1;2)}^{[k]})\cong H^*(G_k,\C),$$
where the latter is the cohomology ring of the Grassmannian of $k$-dimensional 
complex planes in $\C^N$. This shows that 
$\EXT(\widehat{L}_{(1;2)}^{[k]}, \widehat{L}_{(1;2)}^{[k]})$ has dimension one in 
degree zero. Therefore, the identity is a primitive idempotent, 
which means that $\widehat{L}_{(1;2)}^{[k]}$ is indecomposable.
\end{proof}

Consider the following webs, where we assume that $k_1,k_2,k_3\geq 0$ and 
$k_3=k_1+k_2$:
\begin{figure}[htb]
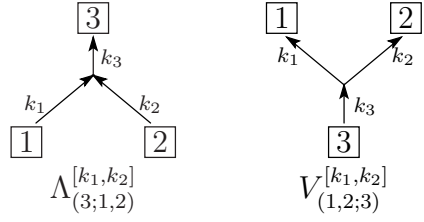

$$
\txt{\input{figure/figgluing-in-3valent-mf}}
\qquad\txt{\input{figure/figgluing-out-3valent-mf}} 
$$
\vskip0.9cm
\caption{${\Lambda}_{(3;1,2)}^{[k_1,k_2]}$ and ${V}^{[k_1,k_2]}_{(1,2;3)}$}
\end{figure}
\begin{definition}
\indent
We define the matrix factorizations  
\begin{equation}
\label{n-mf}
\widehat{\Lambda}_{(3;1,2)}^{[k_1,k_2]} :=\mathop{\boxtimes}_{a=1}^{k_{3}} 
K\Big( P_{a,(3;1,2)}^{[k_1,k_2]} ;X_{a,(3)}^{(k_3)}-X^{(k_1,k_2)}_{a,(1,2)} \Big)_{R^{(k_1,k_2,k_3)}_{(1,2,3)}} \{ - k_1 k_2 \}
\end{equation}
where 
\begin{gather*}
P_{a,(3;1,2)}^{[k_1,k_2]}=\\[0.5em]
\dfrac{
P_{N+1}(...,X^{(k_1,k_2)}_{a-1,(1,2)} ,X_{a,(3)}^{(k_3)},X_{a+1,(3)}^{(k_3)},...)
-P_{N+1}(...,X^{(k_1,k_2)}_{a-1,(1,2)},X^{(k_1,k_2)}_{a,(1,2)},X_{a+1,(3)}^{(k_3)},...)
}
{
X_{a,(3)}^{(k_3)}-X^{(k_1,k_2)}_{a,(1,2)}
},
\end{gather*}
and 
\begin{equation}
\label{v-mf}
\widehat{V}^{[k_1,k_2]}_{(1,2;3)}:=\mathop{\boxtimes}_{a=1}^{k_{3}}
K\Big( P_{a,(1,2;3)}^{[k_1,k_2]} ;X^{(k_1,k_2)}_{a,(1,2)}-X_{a,(3)}^{(k_3)} \Big)_{R^{(k_1,k_2,k_3)}_{(1,2,3)}}
\end{equation}
where 
\begin{gather*}
P_{a,(1,2;3)}^{[k_1,k_2]}=\\[0.5em]
\dfrac{
P_{N+1}(...,X_{a-1,(3)}^{(k_3)} ,X^{(k_1,k_2)}_{a,(1,2)},X^{(k_1,k_2)}_{a+1,(1,2)},...)
-P_{N+1}(...,X_{a-1,(3)}^{(k_3)},X_{a,(3)}^{(k_3)},X^{(k_1,k_2)}_{a+1,(1,2)},...)
}{
X^{(k_1,k_2)}_{a,(1,2)}-X_{a,(3)}^{(k_3)}
}.
\end{gather*}
\end{definition}
\vskip0.5cm
The proof of the following proposition is analogous to the one of 
Proposition~\ref{str-mf-ind1}.
\begin{proposition}\label{str-mf-ind2}
We have the following results:
\begin{enumerate}
\item $\widehat{\Lambda}_{(3;1,2)}^{[k_1,k_2]}$ and $\widehat{V}_{(1,2;3)}^{[k_1,k_2]}$ are homotopy equivalent to the zero matrix factorization if $k_3\geq N+1$;
\vskip0.2cm
\item $\widehat{\Lambda}_{(3;1,2)}^{[k_1,k_2]}$ and $\widehat{V}_{(1,2;3)}^{[k_1,k_2]}$ 
are indecomposable for $0\leq k_3\leq N$.
\end{enumerate}
\end{proposition}
The following proposition can be proved by direct computation. 
\begin{proposition}\label{str-mf-isom}
We have the following isomorphisms:
\begin{eqnarray}
\widehat{\Lambda}_{(3;1,2)}^{[0,k_2]}\simeq\widehat{L}_{(3;2)}^{[k_2]},&\widehat{\Lambda}_{(3;1,2)}^{[k_1,0]}\simeq\widehat{L}_{(3;1)}^{[k_1]};\\[0.5cm]
\widehat{V}_{(1,2;3)}^{[0,k_2]}\simeq\widehat{L}_{(2;3)}^{[k_2]},&
\widehat{\Lambda}_{(1,2;3)}^{[k_1,0]}\simeq\widehat{L}_{(1;3)}^{[k_1]}.
\end{eqnarray}
\end{proposition}

Tensoring the matrix factorizations above, we can associate a 
matrix factorizations $\hat{u}$ to any monomial $\mathfrak{sl}_N$-web from 
$\vec{k}$ to $\vec{k}'$ without tags (but with oriented $N$-colored edges). 
Note that for cups and caps we 
use the same matrix factorization as for $L^{[\vec{k}]}_{(1;2)}$. We also have 
$$
\hat{u}_{\bullet}\cong\widehat{u^*}\{d(\vec{k})\}\langle 1\rangle,
$$
for any monomial web $u\in W(\vec{k},N)$. By Lemma~\ref{lem:ExtvH}, this 
implies that 
\begin{equation}
\label{eq:ExtvH}
\mathrm{EXT}(\hat{u},\hat{v})\cong H(\hat{u}_{\bullet}\btime{R^{\vec{k}}}\hat{v})
\cong H(\widehat{u^*v})\{d(\vec{k})\} \langle 1\rangle 
\end{equation}
for any $u,v\in W(\vec{k},N)$.

For the proof of the following theorem, we refer to Sections 6 through 11 
in~\cite{wu} and Section 3 in~\cite{yo1}.  
\begin{theorem}[Wu, Yonezawa]
\label{thm:catwebrels}
The matrix factorizations associated to webs without tags satisfy all 
relations in Definition~\ref{def:webrelations}, except the first one, up 
to homotopy equivalence. These equivalences are $q$-degree preserving, 
but might involve homological degree shifts. 
\end{theorem}




\section{Categorified quantum $\mathfrak{sl}_m$ and $2$-representations}
\label{sec:catgroupandrep}
\subsection{Categorified $\U{m}$}\label{kl}
Khovanov and Lauda introduced diagrammatic 2-categories $\u(\mathfrak{g})$ 
which categorify the integral version of the corresponding 
idempotented quantum groups~\cite{kl3}. Independently, Rouquier 
also introduced similar 2-categories~\cite{rou}.
Subsequently, Cautis and Lauda~\cite{cl} defined diagrammatic 2-categories 
$\u_Q(\mathfrak{g})$ with implicit scalers $Q$ consisting of $t_{ij}$, $r_i$ and 
$s_{ij}^{pq}$ which determine certain signs in the definition of 
the categorified quantum groups.
\\
\indent
In this section, we recall $\u_Q(\mathfrak{sl}_m)$ briefly and choose the 
implicit scalars $Q$ to be given by $t_{ij}=-1$ if $j=i+1$, $t_{ij}=1$ otherwise, 
$r_i=1$ and $s_{ij}^{pq}=0$. This corresponds precisely to the signed 
version in~\cite{kl3,kl4}. The other conventions here 
are the same as those in Section~\ref{sec:fund}.  

\begin{definition}[Khovanov-Lauda]\label{def:KL}
The $2$-category $\u_Q(\mathfrak{sl}_{m})$ is defined as follows:
\begin{itemize}
\item[$\bullet$] The objects in $\u_Q(\mathfrak{sl}_m)$ are the weights $\l \in  \Z^{m-1} $.
\end{itemize}
For any pair of objects $\l$ and $\l'$ in $\u_Q(\mathfrak{sl}_m)$, the hom category 
$\u_Q(\mathfrak{sl}_m)(\l,\l')$ is the graded additive $\C$-linear category consisting of:
\begin{itemize}
\item[$\bullet$] objects ($1$-morphisms in $\u_Q(\mathfrak{sl}_m)$), 
which are finite formal sums of the form 
$\e_{\ui}{\idm}_{\l}\{t\}$ where $t\in \Z$ is the grading shift 
and $\ui$ is a signed sequence such that 
$\l'=\l+\sum_{a=1}^{l}\epsilon_ai_{a}'$.
\item[$\bullet$] morphisms from $\e_{\ui}{\idm}_{\l}\{t\}$ to 
$\e_{\underline{l}}{\idm}_{\l}\{t'\}$ in $\u_Q(\mathfrak{sl}_m)(\l,\l')$ 
($2$-morphisms in $\u_Q(\mathfrak{sl}_m)$) are 
$\C$-linear combinations of diagrams with degree $t'-t$ spanned by 
composites of the following diagrams:
\end{itemize}

\indent
\begin{eqnarray*}
&&\txt{\input{figure/e-dot}}:\e_{+i}\idm_{\l}\to\e_{+i}\idm_{\l}\{a_{ii}\}\hspace{1cm}
\txt{\input{figure/f-dot}}:\e_{+i}\idm_{\l}\to\e_{-i}\idm_{\l}\{a_{ii}\}
\\[1em]
&&\hspace{-.5cm}\txt{\input{figure/e-f-cup}}:\idm_{\l}\to\e_{(-i,+i)}\idm_{\l}\{\l_i+1\}\hspace{.5cm}
\txt{\input{figure/f-e-cup}}:\idm_{\l}\to\e_{(+i,-i)}\idm_{\l}\{-\l_i+1\}
\\[1em]
&&\hspace{-.5cm}\txt{\input{figure/e-f-cap}}:\e_{(-i,+i)}\idm_{\l}\to\idm_{\l}\{\l_i+1\}\hspace{.5cm}
\txt{\input{figure/f-e-cap}}:\e_{(+i,-i)}\idm_{\l}\to\idm_{\l}\{-\l_i+1\}
\\[1em]
&&\hspace{2cm}\txt{\input{figure/e-j-i1}}:\e_{(+i,+l)}\idm_{\l}\to\e_{(+l,+i)}\idm_{\l}\{-a_{il}\}
\\[1em]
&&\hspace{2cm}\txt{\input{figure/f-j-i1}}:\e_{(-i,-l)}\idm_{\l}\to\e_{(-l,-i)}\idm_{\l}\{-a_{il}\}
\end{eqnarray*}
As already remarked, the relations on the $2$-morphisms are 
those of the signed version in~\cite{kl3,kl4}, which we do not recall here 
because we do not need them explicitly in this paper.  
\end{definition}

We recall Khovanov and Lauda's Proposition 1.4 in~\cite{kl3}.
\begin{theorem}[Khovanov-Lauda]
\label{thm:KL}
The linear map 
$$\U{m}\to K_0^q(\u_Q(\mathfrak{sl}_m))$$ 
defined by 
$$q^t E_{\ui}1_{\lambda}\to \e_{\ui}{\idm}_{\l}\{t\}$$
is an isomorphism of algebras. 
\end{theorem}

\subsection{Cyclotomic KLR algebras and $2$-representations}\label{sec:cyclKLR}
Let $\Lambda$ be a dominant $\mathfrak{sl}_m$-weight, 
$V_{\Lambda}$ the irreducible $\U{m}$-module of highest 
weight $\Lambda$ and $P_{\Lambda}$ the set of weights in $V_{\Lambda}$.
 
\begin{definition}[Khovanov-Lauda, Rouquier]\label{def:cyclKLR}
The {\em cyclotomic KLR algebra} $R_{\Lambda}$ is 
the subquotient of $\u_Q(\mathfrak{sl}_m)$ 
defined by the subalgebra of all diagrams with only downward 
oriented strands and right-most region labeled $\Lambda$ 
modded out by the ideal generated by diagrams of the form  
$$
\txt{\input{figure/cyclo-rel}}
$$
\end{definition}
Note that 
$$R_{\Lambda}=\bigoplus_{\mu\in P_{\Lambda}} R_{\Lambda}(\mu),$$
where $R_{\Lambda}(\mu)$ is the subalgebra generated by all diagrams 
whose left-most region is labeled $\mu$. Brundan and Kleshchev proved that 
$R_{\Lambda}$ is finite-dimensional in Corollary 2.2 in~\cite{bk2}.  
We also define 
$$\mathcal{V}_{\Lambda}^p:=R_{\Lambda}-\mathrm{pmod}_{\mathrm{gr}}.$$

Below we will use Cautis and Lauda's language of strong $\mathfrak{sl}_m$ 
$2$-representations (see Definition 1.2 in~\cite{cl}). For a comparison 
with Rouquier's~\cite{rou} definition of a 
Kac-Moody $2$-representation, see Cautis and Lauda's remark (1) below their 
Definition 1.2. Since we always use the same choice of $Q$ in this paper, 
which we specified above, we call the $2$-representations below simply 
{\em strong $2$-representations}.    

In Section 4.4 in~\cite{bk} Brundan and Kleshchev defined a strong 
$\mathfrak{sl}_m$ $2$-representation on $\mathcal{V}_{\Lambda}$, which can 
be restricted to $\mathcal{V}_{\Lambda}^p$.

Brundan and Kleshchev proved the following theorem (Proposition 4.16 and 
Theorem 4.18 in~\cite{bk}), which was 
conjectured by Khovanov and Lauda~\cite{kl1}.  
 
\begin{theorem}[Brundan-Kleshchev]\label{thm:cyclKLR}
There exists an isomorphism  
$$\delta\colon V_{\Lambda}\to K_0^{q}(\mathcal{V}_{\Lambda}^p)$$
of $\U{m}$-modules. 

Moreover, this isomorphism maps intertwines the $q$-Shapovalov form and 
the Euler form. 
\end{theorem} 

Rouquier proved that $R_{\Lambda}$ is the universal categorification of 
$V_{\Lambda}$ in the following sense. For each 
$\mu\in P_{\Lambda}$, let $\mathcal{C}(\mu)$ be a graded Krull-Schmidt 
$\C$-linear category with finite-dimensional hom-spaces. Take 
$$\mathcal{C}_{\Lambda}:=\bigoplus_{\mu\in P_{\Lambda}} C(\mu).$$
For a proof of the following 
result, see Lemma 5.4, Proposition 5.6 and Corollary 5.7 in~\cite{rou}. 
\begin{proposition}[Rouquier's Universality Proposition, additive version]
\label{prop:Rouquier1}
Suppose that 
\begin{itemize}
\item $\mathcal{C}_{\Lambda}$ is a strong 
$2$-representation of $\mathfrak{sl}_{m}$ by $\C$-linear functors;
\item There exists an indecomposable object $V(\Lambda)\in 
\mathcal{C}(\Lambda)$ such that $\mathcal{E}_{+i}V(\Lambda)=0$, 
for any $i=1,\ldots,m-1$, and $\mathrm{End}(V(\Lambda))\cong\mathbb{C}$; 
\item any object in $\mathcal{C}_{\Lambda}$  
is a direct summand of $XV(\Lambda)$, for some $1$-morphism 
$X\in \u_Q(\mathfrak{sl}_{m})$.
\end{itemize} 
Then there exists an equivalence 
$$\mathcal{V}_{\Lambda}^p\to \mathcal{C}_{\Lambda}$$ 
of additive strong $\mathfrak{sl}_m$ $2$-representations.

In particular, we have 
$$V_{\Lambda}\cong K_0^{q}(R_{\Lambda})\cong K_0^{q}(\mathcal{C}_{\Lambda})$$
as $\U{m}$-modules.
\end{proposition}

%
%
%
%
\section{Some morphisms in $\HMF$}\label{sec:somemorphisms}
In this section, we recall some useful morphisms between matrix 
factorizations. 
\vskip0.5cm
As before, let $R$ be a graded polynomial ring and $p_a$, $q_a$ ($a=1,...,l$) 
be polynomials in $R$. Consider the Koszul matrix factorization 
$$
\btime{\,\,R}_{a=1}^{l}K(p_a;q_a)_{R}.
$$
Recall the following result which we use frequently in the following 
sections. The proof follows directly from the definitions.   
\begin{proposition}\label{hom-eq}
Multiplication by $p_a$ or $q_a$ defines an endomorphism of 
$\btime{\,\,R}_{a=1}^{l}K(p_a;q_a)_{R}$ which is homotopic to $0$. 
\end{proposition}
%
%
%
%
\subsection{Morphism in $\HMF$ (1)}\label{sec-mor1}
We consider the following diagrams:
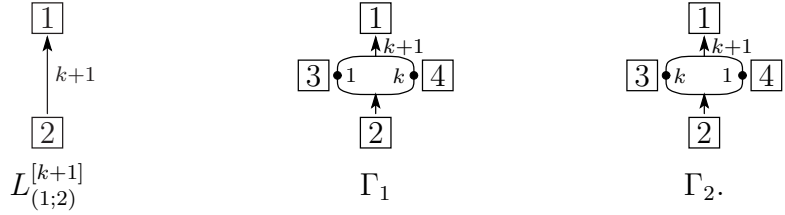
\begin{center}
$
\begin{array}{ccc}
\input{figure/diff1}&\hspace{3cm}\input{figure/diff2}&\hspace{3cm}\input{figure/diff3}\\
L_{(1;2)}^{[k+1]}\hspace{3cm}&\hspace{3cm}\Gamma_1&\hspace{3cm}\Gamma_2.
\end{array}
$
\end{center}
The matrix factorization $\widehat{\Gamma}_1$ is isomorphic to 
\begin{gather}
\label{isom-bubble}
\widehat{\Lambda}_{(1;3,4)}^{[1,k]}\btime{R_{(3,4)}^{(1,k)}}\hspace{-0.1cm} 
\widehat{V}_{(3,4;2)}^{[1,k]}\simeq\\
\widehat{L}_{(1;2)}^{[k+1]}\otimes_{R_{(1,2)}^{(k+1,k+1)}}\left(R_{(1,2,3,4)}^{(k+1,k+1,1,k)}
/J_{(3,4)}^{(1,k)}\right)\{-k\}
\simeq\left(\widehat{L}_{(1;2)}^{[k+1]}\right)^{\oplus [k+1]_q}
\end{gather}
where $J_{(3,4)}^{(1,k)}=\left<X_{1,(3,4)}^{(1,k)},...,X_{k+1,(3,4)}^{(1,k)}\right>$.
We have two $R_{(1,2)}^{(k,k)}$-module morphisms of degree $-k$
\begin{eqnarray}
\nonumber
\xymatrix@R=.5pc{
I_{(4,3)} :& R_{(1,2)}^{(k,k)} \ar[r]\ar@{}[d]|-{\rotatebox{90}{$\in$}}& R_{(1,2,3,4)}^{(k+1,k+1,1,k)}\left/J_{(3,4)}^{(1,k)}\{-k\}\right.\ar@{}[d]|-{\rotatebox{90}{$\in$}} \\
&1\ar@{|->}[r]&1\\ \\
D_{(4,3)} :& R_{(1,2,3,4)}^{(k+1,k+1,1,k)}\left/J_{(3,4)}^{(1,k)}\{-k\}\right.\ar[r]\ar@{}[d]|-{\rotatebox{90}{$\in$}}& R_{(1,2)}^{(k,k)}\ar@{}[d]|-{\rotatebox{90}{$\in$}}\\
&f\ar@{|->}[r]&\partial_{t_{1,4}t_{1,3}}\partial_{t_{2,4}t_{1,3}}...\partial_{t_{k,4}t_{1,3}}f ,
}
\end{eqnarray}
where $\partial_{t_{j,4}t_{1,3}}g(t_{1,3},t_{1,4},...,t_{j,4},...,t_{k,4})$ is defined by
$$
\frac{g(t_{1,3},...,t_{j,4},...)-g(t_{j,4},...,t_{1,3},...)}{t_{j,4}-t_{1,3}}.
$$
Using \eqref{isom-bubble}, we can extend 
the maps ${I_{(4,3)}}$ and ${D_{(4,3)}}$ to morphisms of 
matrix factorizations of degree $-k$
\begin{eqnarray}
\widehat{I}_{(1,k)}&:&\widehat{L}_{(1;2)}^{[k+1]}\longrightarrow\widehat{\Gamma_1},\\
\widehat{D}_{(1,k)}&:&\widehat{\Gamma_1}\longrightarrow\widehat{L}_{(1;2)}^{[k+1]}.
\end{eqnarray}
Since $\Gamma_2$ is symmetric to $\Gamma_1$, we also have
\begin{eqnarray}
\widehat{I}_{(k,1)}&:&\widehat{L}_{(1;2)}^{[k+1]}\longrightarrow\widehat{\Gamma_2},\\
\widehat{D}_{(k,1)}&:&\widehat{\Gamma_2}\longrightarrow\widehat{L}_{(1;2)}^{[k+1]}.
\end{eqnarray}
%
%
%
%
\subsection{Morphisms in $\HMF$ (2)}\label{sec-mor2}
We consider the following diagrams:
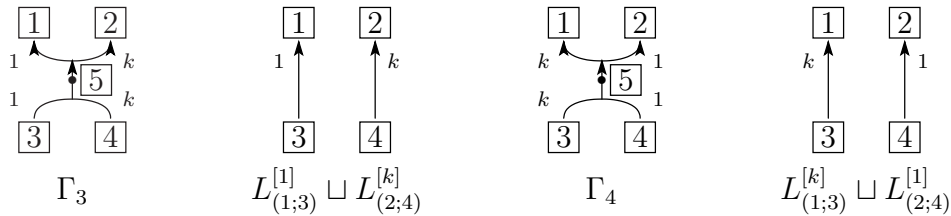
\begin{center}
$
\begin{array}{cccc}
\input{figure/zip2}\hspace{1.5cm}&\hspace{1.5cm}\input{figure/zip1}\hspace{1.5cm}&\hspace{1.5cm}\input{figure/zip4}\hspace{1.5cm}&\hspace{1.5cm}\input{figure/zip3}\hspace{1.5cm}\\
\Gamma_3\hspace{1.5cm}&\hspace{1.5cm}L^{[1]}_{(1;3)}\sqcup L^{[k]}_{(2;4)}\hspace{1.5cm}&\hspace{1.5cm}\Gamma_4\hspace{1.5cm}&\hspace{1.5cm}L^{[k]}_{(1;3)}\sqcup L^{[1]}_{(2;4)}\hspace{1.5cm}
\end{array}
$
\end{center}
We have 
\begin{gather}
\label{eq:Gamma31}
\widehat{\Gamma}_3=\widehat{\Lambda}_{(5;3,4)}^{[1,k]}\btime{R_{(5)}^{(k+1)}} \widehat{V}_{(1,2;5)}^{[1,k]}\simeq\\
\nonumber\widehat{S}_{(1,2;3,4)}\btime{R_{(1,2,3,4)}^{(1,k,1,k)}} K(p_{k+1};(t_{1,1}-t_{1,3})X_{k,(2,3)}^{(k,-1)}))_{R_{(1,2,3,4)}^{(1,k,1,k)}}\{-k\}
\end{gather}
and
\begin{equation}
\label{eq:LkL1}
\widehat{L}_{(1;3)}^{[1]}\btime{\C} \widehat{L}_{(2;4)}^{[k]}
\simeq \widehat{S}_{(1,2;3,4)}\btime{R_{(1,2,3,4)}^{(1,k,1,k)}} K(p_{k+1}X_{k,(2,3)}^{(k,-1)};(t_{1,1}-t_{1,3}))_{R_{(1,2,3,4)}^{(1,k,1,k)}}
\end{equation}
where
\begin{equation*}
\widehat{S}_{(1,2;3,4)}=\quad\boxtimes_{a=1}^{k}K(p_a;(t_{1,1}-t_{1,3})X_{a-1,(2,3)}^{(k,-1)}+x_{a,2}-x_{a,4})_{R_{(1,2,3,4)}^{(1,k,1,k)}}
\end{equation*}
with
\begin{eqnarray*}
p_a&=&\dfrac{P_{N+1}(...,X_{a-1,(3,4)}^{(1,k )},X_{a,(1,2)}^{(1,k )},...)-P_{N+1}(...,X_{a,(3,4)}^{(1,k)},X_{a+1,(1,2)}^{(1,k)},...)}{X_{a,(1,2)}^{(1,k)}-X_{a,(3,4)}^{(1,k)}}\\
&+&t_{1,3}\frac{P_{N+1}(...,X_{a,(3,4)}^{(1,k )},X_{a+1,(1,2)}^{(1,k )},...)-P_{N+1}(...,X_{a+1,(3,4)}^{(1,k)},X_{a+2,(1,2)}^{(1,k)},...)}{X_{a+1,(1,2)}^{(1,k)}-X_{a+1,(3,4)}^{(1,k)}}\\
\end{eqnarray*}
for $1\leq a\leq k$, and 
\begin{eqnarray*}
p_{k+1}&=&\dfrac{P_{N+1}(...,X_{k,(3,4)}^{(1,k )},X_{k+1,(1,2)}^{(1,k )})-P_{N+1}(...,X_{k,(3,4)}^{(1,k)},X_{k+1,(3,4)}^{(1,k)})}{X_{k+1,(1,2)}^{(1,k)}-X_{k+1,(3,4)}^{(1,k)}}.
\end{eqnarray*}

By~\eqref{eq:Gamma31} and~\eqref{eq:LkL1}, 
the morphisms of degree $k$
\begin{eqnarray*}
(1,X_{k(2,3)}^{(k,-1)})&:&K(p_{k+1};(t_{1,1}-t_{1,3})X_{k,(2,3)}^{(k,-1)}))_{R_{(1,2,3,4)}^{(1,k,1,k)}}\{-k\}\\
&&\to K(p_{k+1}X_{k,(2,3)}^{(k,-1)};t_{1,1}-t_{1,3}))_{R_{(1,2,3,4)}^{(1,k,1,k)}}\\
(X_{k(2,3)}^{(k,-1)},1)&:&K(p_{k+1}X_{k,(2,3)}^{(k,-1)};t_{1,1}-t_{1,3}))_{R_{(1,2,3,4)}^{(1,k,1,k)}}\\
&&\to K(p_{k+1};(t_{1,1}-t_{1,3})X_{k,(2,3)}^{(k,-1)}))_{R_{(1,2,3,4)}^{(1,k,1,k)}}\{-k\}.
\end{eqnarray*}
induce morphisms of degree $k$ between $\widehat{\Gamma}_3$ and 
$\widehat{L}_{(1;3)}^{[1]}\btime{\C} \widehat{L}_{(2;4)}^{[k]}$ 
\begin{gather*}
\widehat{U}_{(k,1)}\colon \widehat{\Gamma}_3\xrightarrow{\id_{\widehat{S}}
\boxtimes( 1,X_{k(2,3)}^{(k,-1)})}
\widehat{L}_{(1;3)}^{[1]}\btime{\C} \widehat{L}_{(2;4)}^{[k]}\\[1em]
\widehat{Z}_{(k,1)}\colon 
\widehat{L}_{(1;3)}^{[1]}\btime{\C} \widehat{L}_{(2;4)}^{[k]}
\xrightarrow{\id_{\widehat{S}}\boxtimes ( X_{k(2,3)}^{(k,-1)},1)}
\widehat{\Gamma}_3.
\end{gather*}
Using Proposition \ref{ext-hom}, we get 
\begin{proposition}
\quad\\[-1.5em]
\begin{itemize}
\item[(1)] The morphism $\widehat{U}_{(k,1)}$ induces the 
$R^{(1,k,1,k)}_{(1,2,3,4)}$-linear map in 
$$\EXT^0(\widehat{\Gamma}_3,\widehat{L}_{(1;3)}^{[1]}\btime{\C} 
\widehat{L}_{(2;4)}^{[k]})$$ determined by sending $1$ to $1$.
\item[(2)] The morphism $\widehat{Z}_{(k,1)}$ induces 
the $R^{(1,k,1,k)}_{(1,2,3,4)}$-linear map in 
$$\EXT^0(\widehat{L}_{(1;3)}^{[1]}\btime{\C} \widehat{L}_{(2;4)}^{[k]},
\widehat{\Gamma}_3)$$ determined by multiplying with $X_{k(2,3)}^{(k,-1)}$.
\end{itemize}
\end{proposition}
\vskip0.2cm
Similarly, we get degree $k$ morphisms 
\begin{gather*}
\widehat{U}_{(1,k)}\colon \widehat{\Gamma}_4
\xrightarrow{\id_{\widehat{S}}\boxtimes ( 1,X_{k(1,4)}^{(k,-1)})}
\widehat{L}_{(1;3)}^{[k]}\btime{\C} \widehat{L}_{(2;4)}^{[1]}\\[1em]
\widehat{Z}_{(k,1)}\colon
\widehat{L}_{(1;3)}^{[k]}\btime{\C} \widehat{L}_{(2;4)}^{[1]}
\xrightarrow{\id_{\widehat{S}}\boxtimes ( X_{k(1,4)}^{(k,-1)},1)}
\widehat{\Gamma}_4.
\end{gather*}

%
%
%
%
\subsection{Morphism in $\HMF$ (3)}\label{sec-mor3}
We consider the following diagrams:
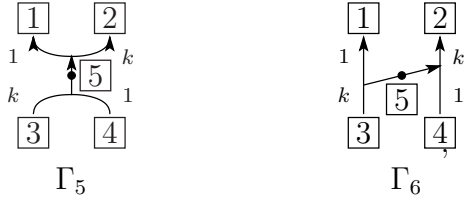
\begin{center}
$
\begin{array}{cc}
\input{figure/twist-zip1}&\hspace{3cm}\input{figure/twist-zip2},\\
\Gamma_5\hspace{3cm}&\hspace{3cm}\Gamma_6
\end{array}
$
\end{center}
We have 
\begin{gather}
\label{eq:Gamma5}
\widehat{\Gamma}_5=\widehat{\Lambda}_{(5;3,4)}^{[k,1]}\btime{R_{(5)}^{(k+1)}} \widehat{V}_{(1,2;5)}^{[1,k]}\simeq\\
\nonumber
\widehat{T}_{(1,2;3,4)}\btime{R_{(1,2,3,4)}^{(1,k,k,1)}} K(q_{k+1};(t_{1,1}-t_{1,4})X_{k,(1,3)}^{(-1,k)}))_{R_{(1,2,3,4)}^{(1,k,k,1)}}\{-k\}\\[1em]
\label{eq:Gamma6}
\widehat{\Gamma}_6=\widehat{\Lambda}_{(2;5,4)}^{[k-1,1]}\btime{R_{(5)}^{(k-1)}} \widehat{V}_{(1,5;3)}^{[1,k-1]}\simeq\\
\nonumber
\widehat{T}_{(1,2;3,4)}\btime{R_{(1,2,3,4)}^{(1,k,k,1)}} K(q_{k+1}(t_{1,1}-t_{1,4});X_{k,(1,3)}^{(-1,k)})_{R_{(1,2,3,4)}^{(1,k,k,1)}}\{-k+1\},
\end{gather}
where
\begin{eqnarray}
\nonumber
\widehat{T}_{(1,2;3,4)}&=&\mathop{\boxtimes}_{a=1}^{k}K(q_a;(t_{1,1}-t_{1,4})X_{a-1,(1,3)}^{(-1,k)}+x_{a,2}-x_{a,3})_{R_{(1,2,3,4)}^{(1,k,k,1)}}\\[1em]
\nonumber
q_a&=&\frac{P_{N+1}(...,X_{a-1,(3,4)}^{(k,1)},X_{a,(1,2)}^{(1,k )},...)-P_{N+1}(...,X_{a,(3,4)}^{(k,1)},X_{a+1,(1,2)}^{(1,k)},...)}{X_{a,(1,2)}^{(1,k)}-X_{a,(3,4)}^{(k,1)}}\\
\nonumber
&&+t_{1,1}\frac{P_{N+1}(...,X_{a,(3,4)}^{(k,1)},X_{a+1,(1,2)}^{(1,k )},...)-P_{N+1}(...,X_{a+1,(3,4)}^{(k,1)},X_{a+2,(1,2)}^{(1,k)},...)}{X_{a+1,(1,2)}^{(1,k)}-X_{a+1,(3,4)}^{(k,1)}}\\
\nonumber
&&(a=1,...,k)\\[1em]
\nonumber
q_{k+1}&=&\frac{P_{N+1}(...,X_{k,(3,4)}^{(k,1 )},X_{k+1,(1,2)}^{(1,k )})-P_{N+1}(...,X_{k,(3,4)}^{(k,1)},X_{k+1,(3,4)}^{(k,1)})}{X_{k+1,(1,2)}^{(1,k)}-X_{k+1,(3,4)}^{(k,1)}}.
\end{eqnarray}
\vskip0.2cm
By the isomorphisms in~\eqref{eq:Gamma5} and~\eqref{eq:Gamma6}, we 
see that the degree $1$ morphisms 
\begin{eqnarray}
\nonumber
(1,t_{1,1}-t_{1,4})&:&K(q_{k+1};(t_{1,1}-t_{1,4})X_{k,(1,3)}^{(-1,k)}))_{R_{(1,2,3,4)}^{(1,k,k,1)}}\\
\nonumber
&&\to K(q_{k+1}(t_{1,1}-t_{1,4});X_{k,(1,3)}^{(-1,k)}))_{R_{(1,2,3,4)}^{(1,k,k,1)}}\{1\}\\
\nonumber
(t_{1,1}-t_{1,4},1)&:&K(q_{k+1}(t_{1,1}-t_{1,4});X_{k,(1,3)}^{(-1,k)}))_{R_{(1,2,3,4)}^{(1,k,k,1)}}\{1\}\\
\nonumber
&&\to K(q_{k+1};(t_{1,1}-t_{1,4})X_{k,(1,3)}^{(-1,k)}))_{R_{(1,2,3,4)}^{(1,k,k,1)}}
\end{eqnarray}
induce morphisms 
\begin{gather*}
\widehat{TU}_{(k,1)}\colon \widehat{\Gamma}_5\to \widehat{\Gamma}_6\\
\widehat{TZ}_{(k,1)}\colon \widehat{\Gamma}_6\to\widehat{\Gamma}_5.
\end{gather*}
\begin{proposition}
\quad\\[-1.5em]
\begin{itemize}
\item[(1)] $\widehat{TU}_{(k,1)}$ corresponds to the 
$R^{(1,k,k,1)}_{(1,2,3,4)}$-linear map in 
$\EXT^0(\widehat{\Gamma}_5,\widehat{\Gamma}_6)$ determined by sending $1$ 
to $1$.
\item[(2)] $\widehat{TZ}_{(k,1)}$ corresponds to the 
$R^{(1,k,k,1)}_{(1,2,3,4)}$-linear map in 
$\EXT^0(\widehat{\Gamma}_6,\widehat{\Gamma}_5)$ 
determined by multiplying with $t_{1,1}-t_{1,4}$. 
\end{itemize}
\end{proposition}
%
%
%
%
\section{The $2$-category $\HMF_{m,d,N}$}\label{sec:HMF}  

\subsection{The matrix factorizations $\oE_{\pm i}$}\label{sec:EFMF}

\noindent
{\bf $\bullet$ In the case of $\vec{k}=(k_1,k_2)$}:\\
\indent
In this case, we have $\alpha=(1,-1)$. 
For $1\leq a\leq k_2$, we define the matrix factorization 
 
\begin{equation}
\oE^{(a)}_{+,[\vec{k}]}:=\widehat{V}_{(2',j;2)}^{[k_2-a,a]}\btime{R_{(j)}^{(a)}}
\widehat{\Lambda}_{(1';j,1)}^{[a,k_1]}\in \HOM_{\HMF_{m,N}}(\vec{k},\vec{k}+a\alpha)
\end{equation}
associated to the diagram in Figure~\ref{e+}.
\begin{figure}[htb]
$$
\input{figure/f-lambda-m-mf}.
$$
\caption{$E^{(a)}_{+,[\vec{k}]}$}\label{e+}
\end{figure}
\\
Similarly, for $1\leq a\leq k_1$, we define the matrix factorization 
\begin{equation}
\oE^{(a)}_{-,[\vec{k}]}:=\widehat{\Lambda}_{(2';2,j)}^{[k_2,a]}\btime{R_{(j)}^{(a)}}
\widehat{V}_{(j,1';1)}^{[a,k_1-a]}\in \HOM_{\HMF_{m,N}}(\vec{k},\vec{k}-a\alpha)
\end{equation}
associated to the diagram in Figure \ref{e-}.
\begin{figure}[htb]
$$
\input{figure/e-lambda-m-mf}.
$$
\caption{$E^{(a)}_{-,[\vec{k}]}$}\label{e-}
\end{figure}
\\
The proof of the following proposition is analogous to that of 
Proposition~\ref{str-mf-ind1}(2).  
\begin{proposition}\label{str-mf-ind3}
The matrix factorizations $\oE^{(a)}_{+,[\vec{k}]}$ and $\oE^{(a)}_{-,[\vec{k}]}$ 
are indecomposable.
\end{proposition}

\noindent
{\bf $\bullet$ The general case $\vec{k}=(k_1,...,k_{m})$}:\\
In this case we have $\alpha_i=\epsilon_i-\epsilon_{i+1}$, for 
$i=1,\ldots,m-1$. We define 
$$
\widehat{\mathbb{1}}_{[\vec{k}]}:=
\widehat{L}^{[k_m]}_{(m';m)}\btime{\,\C}\widehat{L}^{[k_{m-1}]}_{(m-1)';(m-1)}
\btime{\,\C}...\btime{\,\C}\widehat{L}^{[k_2]}_{(2';2)}\btime{\,\C}
\widehat{L}^{[k_1]}_{(1';1)}\in \HOM_{\HMF_{m,N}}(\vec{k},\vec{k})
$$
associated to the diagram in Figure \ref{id}.
\begin{figure}[htb]
$$
\input{figure/id-lambda-mf}
$$
\caption{$\mathbb{1}_{[k]}$}\label{id}
\end{figure}
\\
\indent
For $i=1,...,m-1$ and $1\leq a\leq k_{i+1}$, 
we define the matrix factorization 
\begin{equation}
\oE^{(a)}_{+i,[\vec{k}]}:= \widehat{\mathbb{1}}_{[(k_{i+2},...,k_{m})]}   
\btime{\C}\oE^{(a)}_{+,[(k_i,k_{i+1})]}\btime{\C}\widehat{\mathbb{1}}_{[(k_1,...,k_{i-1})]}
\in \HOM_{\HMF_{m,N}}(\vec{k},\vec{k}+a\alpha_i)
\end{equation}
associated to the diagram in Figure \ref{e+i}.
\begin{figure}[htb]
$$
\input{figure/f-lambda-i-m-mf}
$$
\caption{$E^{(a)}_{+i,[\vec{k}]}$}\label{e+i}
\end{figure}
\\
and, for $1\leq a\leq k_{i}$, the matrix factorization   
\begin{equation}
\oE^{(a)}_{-i,[\vec{k}]}:=\widehat{\mathbb{1}}_{[(k_{i+2},...,k_{m})]}\btime{\C}
\oE^{(a)}_{-,[(k_i,k_{i+1})]}\btime{\C}\widehat{\mathbb{1}}_{[(k_1,...,k_{i-1})]}
\in \HOM_{\HMF_{m,N}}(\vec{k},\vec{k}-a\alpha_i)
\end{equation}
associated to the diagram in Figure \ref{e-i}.
\begin{figure}[htb]
$$
\input{figure/e-lambda-i-m-mf}
$$
\caption{$E^{(a)}_{-i,[\vec{k}]}$}\label{e-i}
\end{figure}
\vskip0.2cm

The proof of the following proposition follows from 
Propositions~\ref{str-mf-ind1}(2) and~\ref{str-mf-ind3}.
\begin{proposition}
$\oE^{(a)}_{+i,[\vec{k}]}$ and $\oE^{(a)}_{-i,[\vec{k}]}$ are indecomposable.
\end{proposition}

Next we have to explain in which order we glue these webs and tensor 
the corresponding matrix factorizations. The logic is determined by 
the fact that these webs determine an action of $\U{m}$ on 
$W_{\Lambda}$ by Proposition~\ref{prop:CKM}), 
which the corresponding matrix factorizations categorify, 
as we will show in Theorem~\ref{thm:2-funct} and 
Section~\ref{sec:cathowe}. 
 
For any signed sequence 
$\underline{i}=(\epsilon_1i_1,\epsilon_2i_2,...,\epsilon_li_l)$ and 
any sequence of non-negative integers $\underline{a}=(a_1,...,a_l)$, let 
$$
(\vec{k},\underline{a},\underline{i})_{j}:=\vec{k}+\sum_{s=j+1}^l \epsilon_s a_s \alpha_s
$$
for $j=1,\ldots,l-1$. By convention, we put 
$$
(\vec{k},\underline{a},\underline{i})_{l}:=\vec{k}.
$$
we define 
\begin{equation}
\oE^{(\underline{a})}_{\ui,[\vec{k}]}:=\mathop{\boxtimes}_{j=1}^{l}
\oE^{(a_j)}_{\epsilon_j i_j,[(\vec{k},\underline{a},\underline{i})_{j}]},
\end{equation}

\begin{example}
Let $n=5$, $\ui=(+3,+1,-2,+4)$ and $\underline{a}=(7,6,5,4)$. Then 
$\oE^{(\underline{a})}_{\ui,[\vec{k}]}$ is the matrix factorization 
associated to the web 
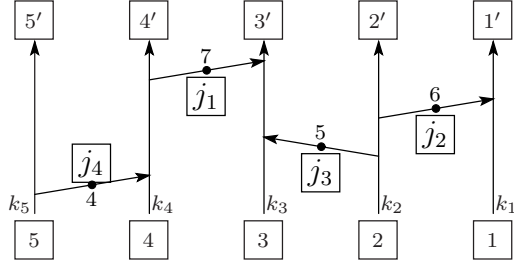
\begin{equation}
\nonumber
\input{figure/fig-ex-gen4}
\end{equation}
\end{example}

\vskip0.5cm

As before, let $N\geq 2$ and $m,d\geq 0$ be arbitrary integers. 
\begin{definition}
We define the $2$-category $\HMF_{m,d,N}$ as follows:
\begin{itemize}
\item The set of objects is $\Lambda(m,d)_N$.
\item For any pair of objects $\vec{k},\vec{k}'\in \Lambda(m,d)_N$, 
we define the hom-category 
$\Hom_{\HMF_{m,d,N}}(\vec{k},\vec{k}')$ to be the full subcategory of
$\HMF_{R^{\vec{k},\vec{k}'}}(P_{N+1}(\mathbb{X}^{\vec{k}})-P_{N+1}(\mathbb{X}^{\vec{k}'}))$
whose set of objects is 
$$
\left\{
\oE^{(\underline{a})}_{\ui,[\vec{k}]} \mid (\vec{k},\underline{a},\underline{i})
=\vec{k}'\right\}.
$$
Horizontal composition is defined by tensoring and vertical composition 
by composing morphisms between matrix factorizations. 
\end{itemize}
\end{definition}

In the rest of this section we define certain $1$ and $2$-morphisms in 
$\HMF_{m,N}^*$ which are crucial for Section~\ref{sec:cathowe}. We always 
assume that $\vec{k}\in \Lambda(m,d)_N$.

\subsection{Some useful $2$-morphisms}
In this section we define some useful 
$2$-morphisms in $\HMF_{m,N}$.
\begin{definition}
We define the endomorphisms of $\oE_{+i,[k]}$ and $\oE_{-i,[k]}$ of degree $2b$, 
denoted $\widehat{t_{1,j}}^b$, by
$$
\widehat{t_{1,j}}^b:=(t_{1,j}^b,t_{1,j}^b)\in \End_{\HMF_{m,N}}(\oE_{\pm i,[k]})
$$
\end{definition}
For the proof of the following proposition see~\cite{wu,yo1,yo2}. 
\begin{proposition}
If $b\geq N$, then $\widehat{t_{1,j}}^b=0$.
\end{proposition}

\begin{figure}[htb]
$$
\xymatrix@R=.5pc{
\txt{\input{figure/e-e-l-l-1}}
\ar@<1mm>[rr]^(.5){\widehat{I}_{(k_{i+1}-1,1)}}
&&\ar@<1mm>[ll]^(.5){\widehat{D}_{(k_{i+1}-1,1)}}
\txt{\input{figure/e-e-l-l-2}}
\ar@<1mm>[r]^(.55){\widehat{\varphi}_{j_1,j_2}}
&\ar@<1mm>[l]^(.45){\widehat{\psi}_{j_1,j_2}}
\txt{\input{figure/e-e-l-l-3}}
\ar@<1mm>[rr]^(.5){(-1)^{k_i}\widehat{Z}_{(1,k_i)}}
&&\ar@<1mm>[ll]^(.5){\widehat{U}_{(1,k_i)}}
\txt{\input{figure/e-e-l-l-4}}\\
\widehat{\mathbb{1}}_{[\vec{k}]}&&\widehat{\Gamma}_2&\widehat{\Gamma}'_2&&
\oE_{(-i,+i),[\vec{k}]}
}
$$
\\[-1em]
\caption{$2$-morphisms I}\label{2-mor1}
\end{figure}
\indent
For the following definition consult Figure~\ref{2-mor1}. The map 
$\hat{\phi}_{j_1,j_2}$ is induced by the homotopy equivalences 
$$\widehat{\Gamma_2}\simeq\widehat{\Gamma}_2'/\langle (t_{1,j_1}-t_{1,j_2})\rangle
\simeq \widehat{\Gamma}_2'$$
and $\hat{\psi}_{j_1,j_2}$ is its inverse. 

\begin{definition}
We define the morphisms 
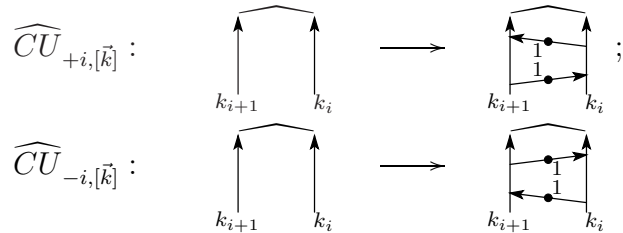
\begin{eqnarray*}
&&\xymatrix{\widehat{CU}_{+i,[\vec{k}]}:&\widehat{\txt{\input{figure/e-e-l-l-1}}}\hspace{.5cm}\ar[r]&\hspace{.5cm}\widehat{\txt{\input{figure/e-e-l-l-4-map}}}};\\
&&\xymatrix{\widehat{CU}_{-i,[\vec{k}]}:&\widehat{\txt{\input{figure/e-e-l-l-1}}}\hspace{.5cm}\ar[r]&\hspace{.5cm}\widehat{\txt{\input{figure/f-f-l-l-4-map}}}}
\end{eqnarray*}
by 
\begin{eqnarray*}
\widehat{CU}_{+i,[\vec{k}]}&:=&(-1)^{k_i}
\widehat{Z}_{(1,k_i)}\widehat{\varphi}_{j_1,j_2}\widehat{I}_{(k_{i+1}-1,1)};\\[1em] 
\widehat{CU}_{-i,[\vec{k}]}&:=&\widehat{Z}_{(k_{i+1},1)}\widehat{\varphi}_{j_1,j_2}
\widehat{I}_{(1,k_{i}-1)}.
\end{eqnarray*}

Their degree is:
\begin{eqnarray*}
\deg(\widehat{CU}_{+i,[\vec{k}]})&=&k_i-k_{i+1}+1,\\
\deg(\widehat{CU}_{-i,[\vec{k}]})&=&-k_i+k_{i+1}+1.
\end{eqnarray*}
\end{definition}
The following lemma follows directly from the definitions and its proof 
is left to the reader. 
\begin{lemma}
$\widehat{CU}_{+i,[\vec{k}]}$ is the $R^{\vec{k}}$-linear map in 
$\EXT^0(\widehat{\mathbb{1}}_{[\vec{k}]},\oE_{(-i,+i),[\vec{k}]})$
determined by multiplying with 
$$(-1)^{k_i}X_{k_i(i,j_2)}^{(k_i,-1)}.$$

Similarly, $\widehat{CU}_{-i,[\vec{k}]}$ is the $R^{\vec{k}}$-linear map 
in $\EXT^0(\widehat{\mathbb{1}}_{[\vec{k}]},\oE_{(+i,-i),[\vec{k}]})$
determined by multiplying with 
$$ X_{k_{i+1}(i+1,j_1)}^{(k_{i+1},-1)}.$$ 
\end{lemma}
\indent
The following definition is similar. 
\begin{definition}
We define 
\begin{eqnarray*}
&&\xymatrix{\widehat{CA}_{+i,[\vec{k}]}:&\widehat{\txt{\input{figure/e-e-l-l-4-map}}}\hspace{.5cm}\ar[r]&\hspace{.5cm}\widehat{\txt{\input{figure/e-e-l-l-1}}}};
\\
&&\xymatrix{\widehat{CA}_{-i,[\vec{k}]}:&\widehat{\txt{\input{figure/f-f-l-l-4-map}}}\hspace{.5cm}\ar[r]&\hspace{.5cm}\widehat{\txt{\input{figure/e-e-l-l-1}}}}
\end{eqnarray*}
by 
\begin{eqnarray*}
\widehat{CA}_{+i,[\vec{k}]}&:=&\widehat{D}_{(k_{i+1}-1,1)}\widehat{\psi}_{j_1,j_2}
\widehat{U}_{(1,k_i)};\\[1em] 
\widehat{CA}_{-i,[\vec{k}]}&:=&(-1)^{k_i}\widehat{D}_{(1,k_i-1)}
\widehat{\psi}_{j_1,j_2}\widehat{U}_{(k_{i+1},1)}.
\end{eqnarray*}

Their degree is:
\begin{eqnarray*}
\deg(\widehat{CA}_{+i,[\vec{k}]})&=&k_i-k_{i+1}+1,\\
\deg(\widehat{CA}_{-i,[\vec{k}]})&=&-k_i+k_{i+1}+1.
\end{eqnarray*}
\end{definition}
\begin{lemma} $\widehat{CA}_{+i,[k]}$ corresponds to the $R^{\vec{k}}$-linear 
map in\\ $\EXT^0(\oE_{(-i,+i),[\vec{k}]},\mathbb{1}_{[\vec{k}]})$ given by 
$$f(t_1,\ldots,t_{k_{i+1}-1},t_{1,j_1},t_{1,j_2})
\mapsto \partial_{t_{1}t_{1,j_1}}\partial_{t_{2}t_{1,j_1}}\cdots\partial_{t_{k_{i+1}-1}t_{1,j_1}}
f_{j_1=j_2}.$$
Here $t_1,\ldots,t_{k_{i+1}-1}$ are the edge variables for the left vertical edge 
inside the square and $f_{j_1=j_2}$ is the polynomial obtained by putting 
$t_{1,j_1}$ equal to $t_{1,j_2}$ in $f$.  
\vskip0.2cm
$\widehat{CA}_{-i,[\vec{k}]}$ corresponds to the $R^{\vec{k}}$-linear map 
in $\EXT^0(\oE_{(+i,-i),[\vec{k}]},\mathbb{1}_{[\vec{k}]})$ given by 
$$f(t_1,\ldots,t_{k_i-1},t_{1,j_1},t_{1,j_2})\mapsto 
\partial_{t_{1,j_1}t_{1}}\partial_{t_{1,j_1}t_{2}}\cdots\partial_{t_{1,j_1}t_{k_i-1}} f_{j_1=j_2}.$$
Here $t_1,\ldots,t_{{k_i}-1}$ are the edge variables for the right vertical 
edge inside the square. 
\end{lemma}
\quad\\[-2em]
\begin{figure}[htb]
$$
\xymatrix@R=.5pc{
\txt{\input{figure/e-e-1-1}}\hspace{1cm}
\ar@<1mm>[r]^(.4){\widehat{\Phi}_1}&
\ar@<1mm>[l]^(.6){\widehat{\Phi}_2}\hspace{1cm}
\txt{\input{figure/e-e-2}}\hspace{1cm}
\ar@<1mm>[r]^(.6){\widehat{D}_{(j_1,j_2)}}&
\ar@<1mm>[l]^(.4){\widehat{I}_{(j_1,j_2)}}\hspace{1cm}
\txt{\input{figure/e-e-2-diff}}\\
\oE_{(-i,-i),[\vec{k}]}\hspace{1cm}&\hspace{1cm}\widehat{\Gamma}_1\hspace{1cm}&\hspace{1cm}\oE^{(2)}_{-i,[\vec{k}]}
}
$$
\caption{$2$-morphisms II}\label{2-mor2}
\end{figure}
\\
\indent
For the following definition, consider Figure~\ref{2-mor2}. We have 
a canonical isomorphism 
$$\widehat{\Phi}_1\colon\oE_{(-i,-i),[\vec{k}]}\to\widehat{\Gamma}_1$$ and 
the morphism 
$$\widehat{D}_{(j_1,j_2)}\colon \widehat{\Gamma}_1\to
\oE^{(2)}_{-i,[\vec{k}]}.$$

We have a similar Figure for $\oE_{(+i,+i),[\vec{k}]}$ and 
$\oE^{(2)}_{+i,[\vec{k}]}$ and corresponding morphisms, for which we use 
the same notation. 
\begin{definition}
We define the morphisms 
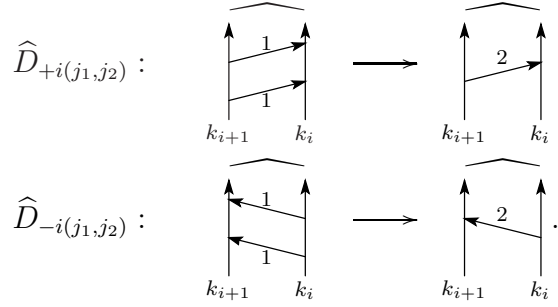
\begin{eqnarray*}
&&\xymatrix{\widehat{D}_{+i(j_1,j_2)}:&\widehat{\txt{\input{figure/f-f-1-1}}}\hspace{.5cm}\ar[r]&\hspace{.5cm}\widehat{\txt{\input{figure/f-f-2-diff}}}}\\
&&\xymatrix{\widehat{D}_{-i(j_1,j_2)}:&\widehat{\txt{\input{figure/e-e-1-1}}}\hspace{.5cm}\ar[r]&\hspace{.5cm}\widehat{\txt{\input{figure/e-e-2-diff}}}}.
\end{eqnarray*}
by 
\begin{eqnarray*}
\widehat{D}_{+i(j_1,j_2)}&:=&\widehat{D}_{(1,1)}\widehat{\Phi}_1\\  
\widehat{D}_{-i(j_1,j_2)}&:=&-\widehat{D}_{(1,1)}\widehat{\Phi}_1.
\end{eqnarray*}

Their degree is:
\begin{eqnarray*}
\deg(\widehat{D}_{+i(j_1,j_2)})=\deg(\widehat{D}_{-i(j_1,j_2)})=-1.
\end{eqnarray*}
\end{definition}
\vskip0.2cm
Let 
$$\widehat{\Phi}_2:\widehat{\Gamma}_1\to\oE_{(-i,-i),[\vec{k}]}$$ 
be the inverse of $\widehat{\Phi}_1$. Recall also the morphism 
$$\widehat{I}_{(j_1,j_2)}:\oE^{(2)}_{-i,[\vec{k}]}\to\widehat{\Gamma}_1.$$ 
\begin{definition}
We define the morphisms 
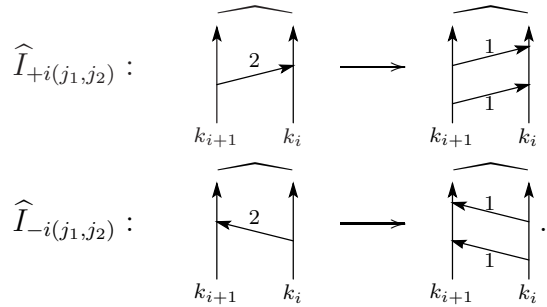
\begin{eqnarray*}
&&\xymatrix{\widehat{I}_{+i(j_1,j_2)}:&\widehat{\txt{\input{figure/f-f-2-diff}}}\hspace{.5cm}\ar[r]&\hspace{.5cm}\widehat{\txt{\input{figure/f-f-1-1}}}}\\
&&\xymatrix{\widehat{I}_{-i(j_1,j_2)}:&\widehat{\txt{\input{figure/e-e-2-diff}}}\hspace{.5cm}\ar[r]&\hspace{.5cm}\widehat{\txt{\input{figure/e-e-1-1}}}}.
\end{eqnarray*}
by
\begin{eqnarray*}
\widehat{I}_{+i(j_1,j_2)}&:=&\widehat{\Phi}_2\widehat{I}_{(1,1)}\\
\widehat{I}_{-i(j_1,j_2)}&:=&\widehat{\Phi}_2\widehat{I}_{(1,1)}.
\end{eqnarray*}

Their degree is:
\begin{eqnarray}
\deg(\widehat{I}_{+i(j_1,j_2)})=\deg(\widehat{I}_{-i(j_1,j_2)})=-1.
\end{eqnarray}
\end{definition}
\vskip0.2cm
Consider the following figure:
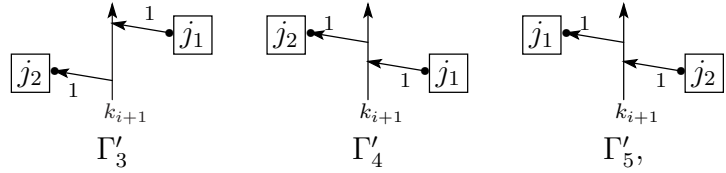
\begin{figure}[htb]
$$
\xymatrix@R=.5pc{
\hspace{-.8cm}\widehat{\scalebox{0.8}{\txt{\input{figure/e-e-1-2-mf1}}}}\hspace{-.8cm}
\ar@<1mm>[r]^(.5){\widehat{TZ}_{(k_{i+1},1)}}&
\ar@<1mm>[l]^(.5){\widehat{TU}_{(k_{i+1},1)}}
\hspace{-.8cm}\widehat{\scalebox{0.8}{\txt{\input{figure/e-e-1-2-mf2}}}}\hspace{-.8cm}
\ar@<1mm>[r]^(.5){\widehat{s}_{j_1,j_2}}&
\ar@<1mm>[l]^(.5){\widehat{s}_{j_1,j_2}}
\hspace{-.8cm}\widehat{\scalebox{0.8}{\txt{\input{figure/e-e-1-2-mf3}}}}\hspace{-.8cm}\\
\oE_{(-i,-(i+1)),[\vec{k}]}&\oE'_{(-(i+1),-i),[\vec{k}]}&\oE_{(-(i+1),-i),[\vec{k}]}
}
$$
\caption{$2$-morphisms III}\label{2-mor3}
\end{figure}
\\
\indent
In 
$$
\begin{array}{ccc}
\input{figure/e-e-1-2-sub1}&\input{figure/e-e-1-2-sub2}&\input{figure/e-e-1-2-sub3}\\\\
\Gamma'_3&\Gamma'_4&\Gamma'_5,
\end{array}
$$
which is part of Figure~\ref{2-mor3}, we have the morphism 
$$\widehat{TZ}_{(k_{i+1},1)}\colon 
\widehat{\Gamma}'_3\to\widehat{\Gamma}'_4$$ 
from Section~\ref{sec-mor3} and the isomorphism 
$$\widehat{s}_{j_1,j_2}\colon\widehat{\Gamma}'_4\to\widehat{\Gamma}'_5$$ 
which swaps the variables $t_{1,j_1}$ and $t_{1,j_2}$.

Of course, there also exists an analogous figure for $\oE_{(i,i+1),[\vec{k}]}$ 
and $\oE_{(i+1,i),[\vec{k}]}$.  
\begin{definition}
We define the morphisms 
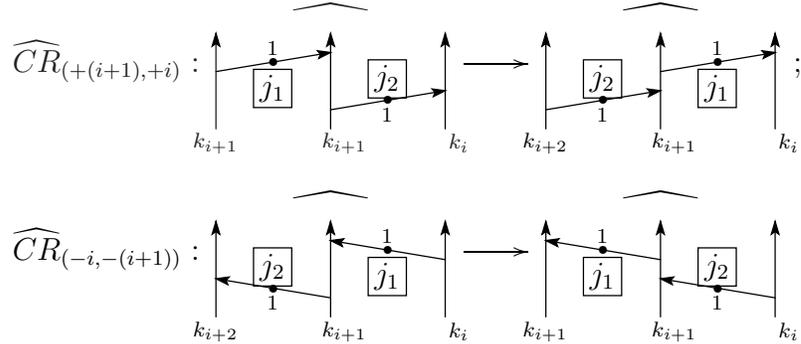
\begin{eqnarray*}
&&\xymatrix{\widehat{CR}_{(+(i+1),+i)}:&\hspace{-2.4cm}\widehat{\txt{\input{figure/f-f-1-2-mf2}}}\hspace{-1.4cm}\ar[r]&\hspace{-1.4cm}\widehat{\txt{\input{figure/f-f-1-2-mf4}}}}\hspace{-1.4cm};\\[1em]
&&\xymatrix{\widehat{CR}_{(-i,-(i+1))}:&\hspace{-2.4cm}\widehat{\txt{\input{figure/e-e-1-2-mf1}}}\hspace{-1.4cm}\ar[r]&\hspace{-1.4cm}\widehat{\txt{\input{figure/e-e-1-2-mf3}}}}\hspace{-1.4cm}
\end{eqnarray*}
by 
\begin{eqnarray*}
\widehat{CR}_{(+(i+1),+i)}&:=&
\widehat{s}_{j_1,j_2}(\widehat{\id}_{\widehat{\Lambda}^{[k_1,1]}_{(k_{i};k'_{i},j_2)}}\boxtimes
\widehat{TZ}_{(1,k_{i+1})}\boxtimes
\widehat{\id}_{\widehat{V}^{[1,k_{i+2}-1]}_{(j_1,k_{i+2};k'_{i+2})}});\\[1em] 
\widehat{CR}_{(-i,-(i+1))}&:=&-\widehat{s}_{j_1,j_2}
(\widehat{\id}_{\widehat{V}^{[k_i-1,1]}_{(k_i,j_1;k'_i)}}\boxtimes
\widehat{TZ}_{(k_{i+1},1)}\boxtimes
\widehat{\id}_{\widehat{\Lambda}^{[1,k_{i+1}]}_{(k_{i+2};j_2,k'_{i+2})}}).
\end{eqnarray*}

Their degree is:
$$
\deg(\widehat{CR}_{(+(i+1),+i)})=\deg(\widehat{CR}_{(-i,-(i+1))})=1.
$$
\end{definition}
\vskip0.2cm
\begin{lemma}
$\widehat{CR}^{[k]}_{+(i+1),+i}$ corresponds to the $R^{\vec{k}}$-linear map 
in 
$$\EXT^0(\oE_{(+(i+1),+i),[\vec{k}]},\oE_{(+i,+(i+1)),[\vec{k}]})$$ 
determined by 
$$f\mapsto \widehat{s}_{j_1,j_2}(t_{1,j_2}-t_{1,j_1})f.$$

$\widehat{CR}^{[k]}_{-i,-(i+1)}$  corresponds to the $R^{\vec{k}}$-linear map 
in 
$$\EXT^0(\oE_{(-i,-(i+1)),[\vec{k}]},\oE_{(-(i+1),-i),[\vec{k}]})$$ 
determined by 
$$f\mapsto -\widehat{s}_{j_1,j_2}(t_{1,j_2}-t_{1,j_1})f.$$
\end{lemma}
\vskip0.2cm
Still with respect to Figure \ref{2-mor3} and the subsequent Figure, we also 
have the morphism 
$$\widehat{TU}_{(k_{i+1},1)}:\widehat{\Gamma}'_4\to\widehat{\Gamma}'_3$$
from Section~\ref{sec-mor3}. 
\begin{definition}
We define the morphisms 
\begin{eqnarray*}
&&\xymatrix{\widehat{CR}_{(+i,+(i+1)))}:&\hspace{-2.4cm}\widehat{\txt{\input{figure/f-f-1-2-mf4}}}\hspace{-1.4cm}\ar[r]&\hspace{-1.4cm}\widehat{\txt{\input{figure/f-f-1-2-mf2}}}}\hspace{-1.4cm}\\[1em]
&&\xymatrix{\widehat{CR}_{(-(i+1),-i)}:&\hspace{-2.4cm}\widehat{\txt{\input{figure/e-e-1-2-mf3}}}\hspace{-1.4cm}\ar[r]&\hspace{-1.4cm}\widehat{\txt{\input{figure/e-e-1-2-mf1}}}}\hspace{-1.4cm}
\end{eqnarray*}
by 
\begin{eqnarray*}
\widehat{CR}_{(+i,+(i+1))}&:=&
\widehat{s}_{j_1,j_2}(\widehat{\id}_{\widehat{\Lambda}^{[k_1,1]}_{(k_{i};k'_{i},j_2)}}\boxtimes
\widehat{TU}_{(1,k_{i+1})}\boxtimes
\widehat{\id}_{\widehat{V}^{[1,k_{i+2}-1]}_{(j_1,k_{i+2};k'_{i+2})}});\\[1em]
\widehat{CR}_{(-(i+1),-i)}&:=&\widehat{s}_{j_1,j_2}
(\widehat{\id}_{\widehat{V}^{[k_i-1,1]}_{(k_i,j_1;k'_i)}}\boxtimes
\widehat{TU}_{(k_{i+1},1)}\boxtimes\widehat{\id}_{\widehat{\Lambda}^{[1,k_{i+1}]}_{(k_{i+2};j_2,k'_{i+2})}}).
\end{eqnarray*}

Their degree is:
$$
\deg(\widehat{CR}_{(+i,+(i+1))})=\deg(\widehat{CR}_{(-(i+1),-i)})=1.
$$
\end{definition}
\vskip0.2cm
\begin{lemma}
$\widehat{CR}^{[k]}_{(+i,+(i+1))}$ corresponds to the $R^{\vec{k}}$-linear map in 
$$\EXT^0(\oE^{(1,1)[k]}_{(+i,+(i+1))},\oE^{(1,1)[k]}_{(+(i+1),+i)})$$
determined by 
$$1\mapsto 1.$$

$\widehat{CR}^{[k]}_{(-(i+1),-i)}$ corresponds to the $R^{\vec{k}}$-linear map in 
$$\EXT^0(\oE^{(1,1)[k]}_{(-(i+1),-i)},\oE^{(1,1)[k]}_{(-i,-(i+1))})$$ 
determined by 
$$
1\mapsto 1.
$$
\end{lemma}
%
%
%
%
\section{Categorified skew Howe duality}
\label{sec:cathowe}
In this section we categorify the quantum skew Howe duality which 
was explained in Section~\ref{sec:Howe}. We first define a $2$-functor 
$\Gamma_{m,d,N}\colon\u_Q(\mathfrak{sl}_m)^*\to\HMF_{m,d,N}^*$ in 
Subsection~\ref{sec:2rep} and then 
use it to categorify the results in Theorem~\ref{thm:ckm} and 
Corollary~\ref{cor:webirrep} in Subsection~\ref{sec:webcat}.
\subsection{Definition of the $2$-representation}
\label{sec:2rep}
Let $m,d,N\geq 2$ be arbitrary integers, until further notice.  
\begin{definition}
\label{def:2rep}
We define a $2$-functor 
$$\Gamma_{m,d,N}\colon\u_Q(\mathfrak{sl}_m)^*\to\HMF_{m,d,N}^*$$ 
which sends:
\begin{itemize}
\item the objects $\lambda=(\lambda_1,...,\lambda_{m-1})\in \Z^{m-1}$ to 
$\vec{k}$, if $\phi_{m,d,N}(\lambda)=\vec{k}\in \Lambda(m,d)_N$, 
or else to zero.  
\item the $1$-morphisms to matrix factorizations:
$$
\xymatrix@R=.3pc{
{\bf 1}_{\lambda}\{t\} \ar@{|->}[r]&
{\left\{
\begin{array}{cl}
\mathbb{1}_{[\vec{k}]}\{t\}&{\rm if}\,\phi_{m,d,N}(\lambda)\in \Lambda(m,d)_N\\
0&{\rm otherwise.}
\end{array}
\right.
}\\
\mathcal{E}_{+i}{\bf 1}_{\lambda}\{t\} \ar@{|->}[r]&
{\left\{
\begin{array}{cl}
\oE_{+i,[\vec{k}]}\{t+k_i-k_{i+1}+1\}&
{\rm if}\,\phi_{m,d,N}(\lambda)\in \Lambda(m,d)_N\\
0&{\rm otherwise.}
\end{array}
\right.
}\\
\mathcal{E}_{-i}{\bf 1}_{\lambda}\{t\} \ar@{|->}[r]&
{\left\{
\begin{array}{cl}
\oE_{-i,[\vec{k}]}\{t-k_i+k_{i+1}+1\}&{\rm if}\,\phi_{m,d,N}(\lambda)\in \Lambda(m,d)_N\\
0&{\rm otherwise.}
\end{array}
\right.
}
}
$$
In general, the $1$-morphism 
${\bf 1}_{\lambda'}\mathcal{E}_{\ui}{\bf 1}_{\lambda}\{t\}$ is mapped to
$$
{\left\{
\begin{array}{cl}
\oE_{\ui,[\vec{k}]}\{t+d(\vec{k})-d(\vec{k}')\}&{\rm if}\,\phi_{m,d,N}(\lambda)\in \Lambda(m,d)_N\\
0&{\rm otherwise.}
\end{array}
\right.
}
$$
with $\vec{k}=\phi_{m,d,N}(\lambda)$ and $\vec{k}'=
\phi_{m,d,N}(\lambda')$.

In order to avoid cluttering of notation, we will write 
$$
\oE_{\ui,[\vec{k}]}':=\oE_{\ui,[\vec{k}]}\{d(\vec{k})-d(\vec{k}')\}=
\Gamma_{m,d,N}(\mathcal{E}_{\ui}{\bf 1}_{\lambda}).
$$
Of course, we have 
$$
\mathbb{1}_{[\vec{k}]}'=\mathbb{1}_{[\vec{k}]}
$$
for any $\vec{k}\in \Lambda(m,d)_N$.
\item the generating $2$-morphisms in $\u_Q(\mathfrak{sl}_m)^*$ 
to the following 2-morphisms in $\HMF_{m,d,N}^*$:
\begin{eqnarray*}
&&
\xymatrix@R=-1pc{
\txt{\input{figure/e-dot}} \ar@{|->}[r]&
{
\widehat{t_{1,j}}:\oE_{+i,[\vec{k}]}'\to\oE_{+i,[\vec{k}]}' (= t_{1,j}\in \EXT^0)
}\\
&\left({\input{figure/e-mf}}\stackrel{(t_{1,j},t_{1,j})}{\longrightarrow}{\input{figure/e-mf}}\right)
}
\\
&&
\xymatrix@R=-1pc{
\txt{\input{figure/f-dot}} \ar@{|->}[r]&
{\widehat{t_{1,j}}:\oE_{-i,[\vec{k}]}'\to\oE_{-i,[\vec{k}]}'} (= t_{1,j}\in \EXT^0)\\
&\left({\input{figure/f-mf}}\stackrel{(t_{1,j},t_{1,j})}{\longrightarrow}{\input{figure/f-mf}}\right)
}
\end{eqnarray*}

\begin{eqnarray*}
&&
\hspace{-.5cm}
\xymatrix@R=-.8pc{
\txt{\input{figure/e-f-cup}} \ar@{|->}[r]&
{\widehat{CU}^{[\vec{k}]}_{+i}:\mathbb{1}_{[\vec{k}]}\to\oE_{(-i,+i),[\vec{k}]}' 
\quad(=(-1)^{\l_i}\widehat{Z}_{(1,k_{i})} \in \EXT^0)
}\\
&\left({\input{figure/empty-mf}\hspace{-.2cm}}\stackrel{\widehat{I}_{(k_{i+1}-1,1)}}{\longrightarrow}{\hspace{-.2cm}\input{figure/2-mor-e-f}\hspace{-.2cm}}\stackrel{(-1)^{\l_i}\widehat{Z}_{(1,k_{i})}}{\longrightarrow}{\hspace{-.2cm}\input{figure/f-e-i-i-mf}}\right)
}\\
&&
\hspace{-.5cm}
\xymatrix@R=-.8pc{
\txt{\input{figure/f-e-cup}} \ar@{|->}[r]&
{\widehat{CU}^{[\vec{k}]}_{-i}:\mathbb{1}_{[\vec{k}]}\to
\oE_{(+i,-i),[\vec{k}]}'\quad(=\widehat{Z}_{(k_{i+1},1)} \in \EXT^0)
}\\
&\left({\input{figure/empty-mf}\hspace{-.2cm}}\stackrel{\widehat{I}_{(1,k_{i}-1)}}{\longrightarrow}{\hspace{-.2cm}\input{figure/2-mor-f-e}\hspace{-.2cm}}\stackrel{\widehat{Z}_{(k_{i+1},1)}}{\longrightarrow}{\hspace{-.2cm}\input{figure/e-f-i-i-mf}}\right)
}
\\[-1em]
&&
\hspace{-.5cm}
\xymatrix@R=-.8pc{
\txt{\input{figure/e-f-cap}} \ar@{|->}[r]&
{\widehat{CA}^{[\vec{k}]}_{+i}:\oE_{(-i,+i),[\vec{k}]}'\to\mathbb{1}_{[\vec{k}]}}
\quad(=\widehat{D}_{(k_{i+1}-1,1)}{\psi}_{j_1,j_2}\in\EXT^0)\\
&\left({\input{figure/f-e-i-i-mf}\hspace{-.2cm}}\stackrel{\widehat{\psi}_{j_1,j_2}\widehat{U}_{(1,k_i)}}{\longrightarrow}{\hspace{-.2cm}\input{figure/2-mor-e-f}\hspace{-.2cm}}\stackrel{\widehat{D}_{(k_{i+1}-1,1)}}{\longrightarrow}
{\hspace{-.2cm}\input{figure/empty-mf}}\right)
}
\\[-1em]
&&
\hspace{-.5cm}
\xymatrix@R=-.8pc{
\txt{\input{figure/f-e-cap}} \ar@{|->}[r]&
{\widehat{CA}^{[\vec{k}]}_{-i}:\oE_{(+i,-i),[\vec{k}]}'\to\mathbb{1}_{[\vec{k}]}}
\quad(=(-1)^{\l_i}\widehat{D}_{(1,k_{i})}{\psi}_{j_1,j_2}\in\EXT^0)\\
&\left({\input{figure/e-f-i-i-mf}\hspace{-.2cm}}
\stackrel{\widehat{\psi}_{j_1,j_2}\widehat{U}_{(k_{i+1},1)}}
{\longrightarrow}{\hspace{-.2cm}\input{figure/2-mor-f-e}\hspace{-.2cm}}
\stackrel{(-1)^{\l_i}\widehat{D}_{(1,k_{i})}}{\longrightarrow}
{\hspace{-.2cm}\input{figure/empty-mf}}\right)
}
\end{eqnarray*}

$$
\hspace{-.5cm}
\xymatrix{
\txt{\input{figure/e-j-i1}} \ar@{|->}[r]&
{
\left\{
\begin{array}{l}
\widehat{I}_{+i,(j_1,j_2)}\widehat{D}_{+i,(j_1,j_2)}:\oE_{(+i,+i),[\vec{k}]}'\to
\oE_{(+i,+i),[\vec{k}]}'\quad\quad\mathrm{if}\, l=i
\\
\quad(=\widehat{D}_{(j_1,j_2)}\in\EXT^0)\\
\left({\input{figure/e-e-i-i-mf}}\stackrel{\widehat{D}_{+i}}{\longrightarrow}{\input{figure/e-2-mf}}\stackrel{\widehat{I}_{+i}}{\longrightarrow}{\input{figure/e-e-i-i-mf}}\right)
\\[2em]
\widehat{CR}_{(+i,+(i+1))}:\oE_{(+i,+(i+1)), [\vec{k}]}'\to
\oE_{(+(i+1),+i),[\vec{k}]}'\hspace{1.5cm}\mathrm{if}\, l=i+1
\\
\quad(=\hat{s}_{(j_1,j_2)}\in\EXT^0)\\
\left({\input{figure/e-e-i+1-i-mf}\hspace{-.5cm}}\stackrel{\widehat{CR}_{(+i,+(i+1))}}{\longrightarrow}{\hspace{-.5cm}\input{figure/e-e-i-i+1-mf}}\right)
\\[2em]
\widehat{CR}_{(+(i+1),+i)}:\oE_{(+i,+(i-1)),[\vec{k}]}'\to\
\oE_{(+(i-1),+i),[\vec{k}]}'\hspace{1.5cm}\mathrm{if}\, l=i-1
\\
\quad(=\hat{s}_{(j_1,j_2)}(t_{1,j_2}-t_{1,j_1})\in\EXT^0)\\
\left({\input{figure/e-e-i-1-i-mf}\hspace{-.5cm}}\stackrel{\widehat{CR}_{(+(i+1),+i)}}{\longrightarrow}{\hspace{-.5cm}\input{figure/e-e-i-i-1-mf}}\right)
\\[2em]
\hat{s}_{j_1,j_2}:\oE_{(+i,+l), [\vec{k}]}'\to\oE_{(+l,+i),[\vec{k}]}'
\quad(=\hat{s}_{(j_1,j_2)}\in\EXT^0)\quad\mathrm{if}\, |l-i|\geq 2
\\
\left({\input{figure/e-e-i-j-mf}}\stackrel{\widehat{s}_{j_1,j_2}}{\longrightarrow}{\input{figure/e-e-j-i-mf}}\right)
\end{array}
\right.
}
}
$$
$$
\hspace{-.5cm}
\xymatrix@R=.3pc{
\input{figure/f-j-i1} \ar@{|->}[r]&
{\left\{
\begin{array}{l}
\widehat{I}_{-i,(j_1,j_2)}\widehat{D}_{-i,(j_1,j_2)}:\oE_{(-i,-i),[\vec{k}]}'\to
\oE_{(-i,-i),[\vec{k}]}'\hspace{1.5cm}\mathrm{if}\, l=i
\\
\quad(=\widehat{D}_{(j_1,j_2)}\in\EXT^0)\\
\left({\input{figure/f-f-i-i-mf}}\stackrel{\widehat{D}_{-i}}{\longrightarrow}{\input{figure/f-2-mf}}\stackrel{\widehat{I}_{-i}}{\longrightarrow}{\input{figure/f-f-i-i-mf}}\right)
\\[1.5em]
\widehat{CR}_{(-i,-(i+1))}:\oE_{(-i,-(i+1)),[\vec{k}]}'\to\oE_{(-(i+1),-i),[\vec{k}]}'
\hspace{1.5cm}\mathrm{if}\, l=i+1\\
\quad(=\hat{s}_{(j_1,j_2)}(t_{1,j_1}-t_{1,j_2})\in\EXT^0)
\\
\left({\input{figure/f-f-i+1-i-mf}\hspace{-.5cm}}\stackrel{\widehat{CR}_{(-i,-(i+1))}}{\longrightarrow}{\hspace{-.5cm}\input{figure/f-f-i-i+1-mf}}\right)
\\[1.5em]
\widehat{CR}_{(-(i+1),-i)}:\oE_{(-(i+1),-i),[\vec{k}]}'\to\oE_{(-i,-(i+1)),[\vec{k}]}'
\hspace{1.5cm}\mathrm{if}\, l=i-1\\
\quad(=\hat{s}_{(j_1,j_2)}\in\EXT^0)
\\
\left({\input{figure/f-f-i-1-i-mf}\hspace{-.5cm}}\stackrel{\widehat{CR}_{(-(i+1),-i)}}{\longrightarrow}{\hspace{-.5cm}\input{figure/f-f-i-i-1-mf}}\right)
\\[1.5em]
\hat{s}_{j_1,j_2}:\oE_{(-i,-l),[\vec{k}]}'\to\oE_{(-l,-i),[\vec{k}]}'
\quad(=\hat{s}_{(j_1,j_2)}\in\EXT^0)\quad\mathrm{if}\, |l-i|\geq 2
\\
\left({\input{figure/f-f-i-j-mf}}\stackrel{\hat{s}_{j_1,j_2}}{\longrightarrow}{\input{figure/f-f-j-i-mf}}\right)
\end{array}
\right.
}
}
$$
\end{itemize}
\end{definition}
\vskip0.5cm
\begin{theorem}
\label{thm:2-funct}
$\Gamma_{m,d,N}\colon \u_Q(\mathfrak{sl}_m)^*\to \HMF_{m,d,N}^*$ is a well-defined 
$2$-functor.
\end{theorem}
\begin{proof}
Our $\Gamma_{m,d,N}$ is very similar to Khovanov and Lauda's $2$-representation 
$$\Gamma_{d}\colon \u_Q(\mathfrak{sl}_m)^*\to \mathrm{\bf Flag}^*_d$$ 
in Section 6 in~\cite{kl3}. The proof that $\Gamma_{m,d,N}$ is well-defined is 
completely analogous to their proof of Theorem 6.9, 
as can be seen by comparing the images of the generating 
$2$-morphisms in $\u_Q(\mathfrak{sl}_m)^*$ for $\Gamma_{m,d,N}$ 
and $\Gamma_d$. We will therefore omit the actual calculations here. 
Note that our sign conventions in the definition of 
$\Gamma_{m,d,N}$ correspond to the ones used in~\cite{msv}.  
\end{proof}

Although $\Gamma_{m,d,N}$ and $\Gamma_d$ are similar and the proof of their 
well-definedness relies on the same calculations, they are not equivalent. 
To explain the difference, let us compare the 
decategorification of both $2$-functors. 
First let us explain the statement in Theorem 6.14 in~\cite{kl3}. 
As explained in~\cite{msv}, the $2$-functor $\Gamma_{d}$ actually 
categorifies the surjective homomorphims of $\U{m}$ onto 
\begin{equation}
\label{eq:schur}
\prod_{\phi_{m,d,N}(\lambda)\in \Lambda(m,d)^+}\mathrm{End}_{\C(q)}(V_{\lambda}),
\end{equation}
which is isomorphic to the quantum Schur algebra $S_q(m,d)$. 
Here 
\begin{gather*}
\Lambda(m,d):=\{\mu\in \N^m\mid \mu_1+\cdots,\mu_m=d\}
\end{gather*}
is the set of $m$-part compositions of $d$ and $\Lambda(m,d)^+$ is 
its subset of partitions, i.e. those 
$\mu\in\Lambda(m,d)$ such that $\mu_1\geq \mu_2\geq \cdots,\mu_m$. 
Thus $\Gamma_d$ descends to a quotient of $\u_Q(\mathfrak{sl}_m)$ which 
categorifies $S_q(m,d)$, as was proved in~\cite{msv}. 

The irreducible $\U{m}$-module $V_d$ with highest weight 
$(d,0,\ldots,0)$ is isomorphic to the left ideal 
$S_q(m,d)1_{(d,0,\ldots,0,0)}\triangleleft S_q(m,d)$ and is categorified by 
Khovanov and Lauda's category 
\begin{equation}
\label{eq:VN}
\bigoplus_{\underline{k}}H_{\underline{k}}-\mathrm{gmod}.
\end{equation} 
So the statement of Theorem 6.14 in~\cite{kl3} should be interpreted as 
meaning that the category in~\eqref{eq:VN} categorifies the underlying 
vector space $V_d$ and that $\Gamma_d$ categorifies the action of $\U{m}$ on 
$V_d$. For more details on the quantum Schur algebra and its categorification 
see~\cite{msv} and references therein.  

Our $\Gamma_{m,d,N}$ categorifies the surjective 
homomorphism of $\U{m}$ onto the quotient of the algebra in~\eqref{eq:schur} 
by the ideal generated by 
those $\mu\in\Lambda(m,d)^+$ which are not $N$-bounded. 

For $m\geq d=N\ell$ and $\Lambda=N\omega_{\ell}$, 
the irreducible $\U{m}$-representation $V_{\Lambda}$ 
can be obtained as (sub)quotient of $S_q(m,d)$, but a different one. 
We have  
$$V_{\Lambda}\cong S_q(m,d)1_{\Lambda}/[\mu > \Lambda],$$
where the ideal $[\mu >\Lambda]$ is generated by all 
$1_{\mu}$ with $\mu > \Lambda$. 

Recall that $V_{\Lambda}\cong W_{\Lambda}$, 
by Corollary~\ref{cor:webirrep}. As we will show in 
Theorem~\ref{thm:categorification1}, 
our $\Gamma_{m,d,N}$ defines an additive strong $\mathfrak{sl}_m$ 
2-representation on $\mathcal{W}^{\circ}_{\Lambda}$, which categorifies the 
$\U{m}$-representation on $W_{\Lambda}$.   

\subsection{The graded web category}
\label{sec:webcat}
Let $m\geq d=N\ell$ and $\Lambda=N\omega_{\ell}$. 
In this section we define a $\C$-linear category 
$\mathcal{W}^{\circ}(\vec{k},N)$. 
\begin{definition}
\label{def:web-category}
The objects of $\mathcal{W}^{\circ}(\vec{k},N)$ are by definition 
all matrix factorizations which are homotopy equivalent to 
direct sums of matrix factorizations of the form $\hat{u}\{t\}$, 
where $u$ is an $N$-ladder with $m$ uprights in $W(\vec{k},N)$ and 
$t\in\Z$. 

For any pair of objects $X,Y\in \mathcal{W}^{\circ}(\vec{k},N)$, 
we define the hom-space between them as 
$$\mathcal{W}^{\circ}(X,Y):=
\mathrm{Ext}(X,Y).$$

Composition in $\mathcal{W}^{\circ}(\vec{k},N)$ is induced by the composition of 
homomorphisms between matrix factorizations. 

Note that $\mathcal{W}^{\circ}(\vec{k},N)^*$ is a $\mathbb{Z}$-graded 
$\C$-linear additive category which admits translation 
and has finite-dimensional 
hom-spaces. 
\end{definition}

By definition, $\mathcal{W}^{\circ}(\vec{k},N)$ is a full subcategory of the 
homotopy category of matrix factorizations with fixed potential 
determined by $\vec{k}$ and $N$. The latter category is Krull-Schmidt 
by Propositions 24 and 25 in~\cite{kr}. Therefore, we can take the 
Karoubi envelope or idempotent completion of 
$\mathcal{W}^{\circ}(\vec{k},N)$, denoted 
$\dot{\mathcal{W}^{\circ}}(\vec{k},N)$, which is also Krull-Schmidt.  
The point here is that a matrix factorization associated to a monomial 
web might have direct summands in the homotopy category 
which are not associated to any monomial webs, so we have to 
include these in our web category by taking the Karoubi envelope. For 
$N\geq 3$ it is very hard to determine the indecomposable summands of 
webs in general, which is why surjectivity in 
Corollary~\ref{cor:categorification11} would 
be hard to prove directly. For more information for $N=3$ 
see~\cite{kk,mpt,mn} and for general $N$ see the sequel to this 
paper~\cite{mack2}.   
\vskip0.5cm
For any monomial web $u\in W(\vec{k},N)$, 
we write 
$$\hat{u}':=\hat{u}\{-d(\vec{k})\}.$$
This is consistent with our notation in Definition~\ref{def:2rep}. 
\begin{definition}
\label{def:gammamap}
We define a linear map 
$$\delta_{\vec{k},N}^{\circ}\colon W(\vec{k},N)\to 
K_0^q(\dot{\mathcal{W}}^{\circ}(\vec{k},N))$$
by 
$$u\mapsto [\hat{u}']=q^{-d(\vec{k})}[\hat{u}],$$
for any $N$-ladder with $m$-uprights $u\in W(\vec{k},N)$. 
\end{definition}
\noindent In Corollary~\ref{cor:categorification11}, we will show that 
$\delta^{\circ}_{\vec{k},N}$ is an isomorphism. For now, we can only show injectivity.  
\begin{lemma}
\label{lem:injectivity}
The map $\delta^{\circ}_{\vec{k},N}$ is injective.
\end{lemma}
\begin{proof}
Recall that the Euler form 
$$\langle [P],[Q]\rangle=\dim_q \mathrm{HOM}(P,Q)$$ 
is a non-degenerate $q$-sesquilinear form on 
$K_0^q(W^{\circ}(\vec{k},N))$. Furthermore, the sesquilinear web form 
gives a non-degenerate $q$-sesquilinear form 
on $W(\vec{k},N)$. 

The map $\delta^{\circ}_{\vec{k},N}$ is an isometry w.r.t. these two forms 
because we have 
$$\dim_q (\mathcal{W}^{\circ}(\hat{u}',\hat{v}'))=
\dim_q(\mathrm{EXT}(\hat{u}',\hat{v}'))=$$
$$q^{d(\vec{k})}\dim_q(H(\widehat{u^*v}))=q^{d(\vec{k})}\mathrm{ev}(u^*v)=
\langle u,v\rangle.$$
for any $N$-ladders with $m$ uprights $u,v\in W(\vec{k},N)$, by Theorem~\ref{thm:catwebrels}. 

Since isometries for non-degenerate forms are always injective, this proves 
the lemma.  
\end{proof}
\begin{definition}
\label{def:catwebspace}
Define 
\begin{eqnarray*}
\mathcal{W}^{\circ}_{\Lambda}&:=&\bigoplus_{\vec{k}\in \Lambda(m,d)_N}
\mathcal{W}^{\circ}(\vec{k},N).
\end{eqnarray*} 
\end{definition}
We will now show that $\Gamma_{m,N}$ induces a strong $\mathfrak{sl}_m$ 
$2$-representation on $\dot{\mathcal{W}}^{\circ}_{\Lambda}$. The idea is 
quite simple. Given an object $1_{\lambda'}\mathcal{E}_{\underline{i}}1_{\lambda}$ in 
$\u_Q(\mathfrak{sl}_m)$ and an object $\hat{w}'\in 
\mathcal{W}^{\circ}(\vec{k},N)$, such that $\phi_{m,d,N}(\lambda)=\vec{k}$ 
and $\phi_{m,d,N}(\lambda')=\vec{k}'$, we define the 
action of $1_{\lambda'}\mathcal{E}_{\underline{i}}1_{\lambda}$ on $\hat{w}'$ by glueing 
the ladder associated to $E_{\underline{i}}1_{\lambda}$ 
(see Proposition~\ref{prop:CKM}) on top of $w$, which gives a ladder again, 
and taking the corresponding matrix factorization 
$$
\hat{E}_{\underline{i},[\vec{k}]}'\btime{R^{\vec{k}}} \hat{w}'\simeq\widehat{E_{\underline{i}}w}'
\in \mathcal{W}^{\circ}(\vec{k}',N).$$ 
This defines the $2$-representation 
on objects. 

We continue to assume that $\phi_{m,d,N}(\lambda)=\vec{k}\in \Lambda(m,d)_N$. 
Given any $2$-morphism $f\in \mathrm{HOM}(\mathcal{E}_{\underline{i}}1_{\lambda}, 
\mathcal{E}_{\underline{j}}1_{\lambda})$ in $\u_Q(\mathfrak{sl}_m)$,  
any pair of objects $\hat{u}',\hat{v}'\in \mathcal{W}^{\circ}(\vec{k},N)$ and any 
morphism $\phi\in \mathcal{W}^{\circ}(\hat{u}',\hat{v}')=
\mathrm{EXT}(\hat{u}',\hat{v}')$, we define 
the action of $f$ on $\phi$ by 
$$
\hat{f}\boxtimes \phi\in \mathrm{EXT}(\hat{E}_{\underline{i},[\vec{k}]}'
\btime{R^{\vec{k}}} \hat{u}', \hat{E}_{\underline{j},[\vec{k}]}'
\btime{R^{\vec{k}}} \hat{v}').
$$    

By Theorem~\ref{thm:2-funct}, 
the definition above gives a well-defined $2$-action on 
$\mathcal{W}^{\circ}_{\Lambda}$ which extends to a well-defined strong 
$\mathfrak{sl}_m$ $2$-representation on 
the Karoubi envelope $\dot{\mathcal{W}}^{\circ}_{\Lambda}$.  

Let us now show that we can apply the additive version of Rouquier's 
universality proposition, reproduced in Proposition~\ref{prop:Rouquier1}, 
to this strong $\mathfrak{sl}_m$ $2$-representation. 

\begin{theorem}
\label{thm:categorification1}
With the definitions above, $\dot{\mathcal{W}}^{\circ}_\Lambda$ is a 
strong $\mathfrak{sl}_m$ $2$-representation equivalent to 
$\mathcal{V}_{\Lambda}^p$. 
\end{theorem}
\begin{proof}
We have to show that the last three conditions in 
Proposition~\ref{prop:Rouquier1} are also fulfilled. 

Take $w_{\Lambda}$ to be 
the web consisting of $\ell$ vertical upward oriented $N$-edges 
in $W(\Lambda,N)$. By 
Corollary~\ref{cor:positivityforclosedwebs}, any indecomposable 
object in $\dot{\mathcal{W}}^{\circ}(\Lambda,N)$ is isomorphic to 
$\hat{w}_{\Lambda}$ up to a grading shift. 

Clearly we have  
$$\mathrm{End}(\hat{w}_{\Lambda})\cong \C.$$ 

And finally, any $N$-ladder with $m$ uprights is the image of a 
product of divided powers by Proposition~\ref{prop:CKM}, so we see 
that the third condition of Proposition~\ref{prop:Rouquier1} is fulfilled. 
\end{proof}

Let $$\delta^{\circ}\colon W_{\Lambda}\to 
K_0^q(\dot{\mathcal{W}}^{\circ}_{\Lambda})$$ 
be the direct sum of the maps $\delta^{\circ}_{\vec{k},N}$ from 
Definition~\ref{def:gammamap} over all $\vec{k}\in \Lambda(m,d)_N$.  
\begin{corollary}
\label{cor:categorification11}
The map $\delta^{\circ}$ is an isomorphism of 
$\U{m}$-representations. In particular, $\delta^{\circ}_{\vec{k},N}$ is 
an isomorphism for all $\vec{k}\in \Lambda(m,d)_N$.
\end{corollary}
\begin{proof}
By Theorem~\ref{thm:categorification1}, we have 
$$\dim_q K_0^q(\dot{\mathcal{W}}^{\circ}_{\Lambda})=
\dim_q K_0^q(\mathcal{V}_{\Lambda})=\dim_q V_{\Lambda}.$$
Since $\dim_q W_{\Lambda}=\dim_q V_{\Lambda}$ by Corollary~\ref{cor:webirrep}, 
we see that 
$$\dim_q K_0^q(\dot{\mathcal{W}}^{\circ}_{\Lambda})=\dim_q W_{\Lambda},$$
which proves this corollary. 
\end{proof}

%
%
%
%

\vskip0.3cm
\noindent M.~M.: {\sl \small CAMGSD, Instituto Superior T\'{e}cnico, 
Lisboa, Portugal; Universidade do Algarve, Faro, Portugal} 
\newline \noindent {\tt \small email: mmackaay@ualg.pt}
\vskip0.2cm
\noindent Y.~Y: {\sl \small Graduate School of Mathematics, 
Nagoya University, 464-8602 Furocho, Chikusaku, Nagoya, Japan}
\newline \noindent {\tt \small email: yasuyoshi.yonezawa@math.nagoya-u.ac.jp}
\end{document}